\documentclass[intlimits,10pt,reqno]{amsart}
\usepackage{amsaddr}
\usepackage{amsmath,amssymb,bezier,amscd,verbatim,bezier,epsfig,bbm}
\usepackage{latexsym}
\usepackage{pgf}
\usepackage{tikz}
\usetikzlibrary{matrix,calc,arrows}
\usepackage[all]{xy}
\newdir{ >}{{}*!/-5pt/\dir{>}}

\newcommand{\nc}{\newcommand}
\newcommand{\al}{\alpha}
\newcommand{\bt}{\beta}

\nc{\vf}{\varphi}
\newcommand{\om}{\omega}
\newcommand{\Om}{\Omega}

\nc{\bC}{\mathbf{C}}

\newcommand{\bE}{{\bf E}}

\newcommand{\Ab}{\operatorname{{\mathbb{A}b}}}
\nc{\AB}{\operatorname{{\mathbb A}b}}
\newcommand{\bH}{\mathbb{H}}

\nc{\bg}{{\mathbb G}}
\nc{\bl}{{\mathbb L}}
\newcommand{\pa}{\partial}

\newcommand{\lra}{\longrightarrow}

\newcommand{\bB}{\mathbb{B}}

\newcommand{\bJ}{\mathbb{J}}

\newcommand{\aaa}{{\bf{\mathfrak{a}}}}
\newcommand{\bbb}{{\bf b}}

\newcommand{\K}{\mathbb{K}}

\nc{\cf}{{\mathcal F}}
\nc{\cF}{{\mathcal F}}

\renewcommand{\Bbb}{\mathbbm} 
\nc{\GR}{\operatorname{{\mathbb G}r}}
\nc{\BL}{{\Bbb L}}
\nc{\BZ}{{\Bbb Z}}
\nc{\grc}{\GR^c}
\nc{\grb}{\GR^\bullet}
\nc{\all}{\allowdisplaybreaks}
\nc{\ogrc}{\ol{\GR}}

\newcommand{\cB}{\mathfrak{B}}

\newcommand{\ol}{\overline}
\nc{\tm}{\times}
\nc{\sps}{\supset}

\nc{\bsl}{\backslash}

\newcommand{\tp}{\otimes}

\newcommand{\Ext}{\operatorname{Ext}}

\newcommand{\Der}{\operatorname{Der}}

\newcommand{\Ker}{\operatorname{Ker}}

\newcommand{\Hom}{\operatorname{Hom}}

\nc{\Imm}{\operatorname{Im}}

\nc{\gr}{\operatorname{{\mathbb{G}r}}} 

\newcommand{\cone}{\operatorname{cone}}
\newcommand{\NP}{\operatorname{NP}}
\newcommand{\NLP}{\operatorname{NLP}}

\newcommand{\inc}{\operatorname{inc}}
\newcommand{\Poiss}{\operatorname{Poiss}}
\newcommand{\AWB}{\operatorname{AWB}}
\newcommand{\USGA}{\operatorname{USGA}}
\newcommand{\Actor}{\operatorname{Actor}}
\nc{\Aut}{\operatorname{Aut}}
\nc{\ab}{\operatorname{{\mathbb{A}b}}} 
\newcommand{\ad}{\operatorname{ad}}
\newcommand{\cdim}{\operatorname{c.dim}}
\nc{\leib}{{\mathbbm{L}\mathrm{eibniz}}}
\nc{\lie}{{\mathbbm{L}\mathrm{ie}}} 
\nc{\grd}{\gr^\bullet} 
\nc{\abd}{\Ab^\bullet}
\nc{\grr}{\gr^{[\;]}}
\nc{\ggr}{G^{[\;]}}

\numberwithin{equation}{section}

\newtheorem{theorem}{Theorem}[section]
\newtheorem{lemma}[theorem]{Lemma}
\newtheorem{proposition}[theorem]{Proposition}

\newtheorem{corollary}[theorem]{Corollary}

\theoremstyle{definition}
\newtheorem{definition}[theorem]{Definition}
\newtheorem{example}[theorem]{Example}

\title[Left-right Noncommutative Poisson algebras] {Left-right Noncommutative Poisson algebras}
\author[J.M. Casas]{Jos\'e~M. Casas}
\address{Department of Applied Mathematics I, University of Vigo, Pontevedra, 36005,  Spain}
\author[T. Datuashvili]{Tamar Datuashvili}
\address{Andrea Razmadze Mathematical Institute at the Ivane Javakhishvili Tbilisi
 State University, University str.2, 0143 Tbilisi, Georgia}
\author[M. Ladra]{Manuel Ladra}
\address{Department of Algebra,  University of Santiago de Compostela, 15782
Santiago de Compostela, Spain}
\email{jmcasas@uvigo.es, tamar@rmi.ge, manuel.ladra@usc.es}

\begin{document}

\begin{abstract}
The notions of left-right noncommutative Poisson algebra
($\NP^{lr}$-algebra) and left-right algebra with bracket
$\AWB^{lr}$ are introduced. These algebras are special cases of $\NLP$-algebras and algebras with bracket $\AWB$,
 respectively, studied earlier. An $\NP^{lr}$-algebra is a noncommutative
analogue of the classical Poisson algebra. Properties of the new
algebras are studied. The constructions of free objects in the
corresponding categories are given. The relations between the
properties of $\NP^{lr}$-algebras, the underlying $\AWB^{lr}$,
associative and Leibniz algebras are investigated.  In the
categories $\AWB^{lr}$ and $\NP^{lr}$-algebras the notions of actions,
representations, centers, actors and crossed modules are described as
special cases of the corresponding well-known notions in categories
of groups with operations. The cohomologies of $\NP^{lr}$-algebras
and $\AWB^{lr}$ (resp. of $\NP^r$-algebras and $\AWB^r$) are defined
and the relations between them and the Hochschild, Quillen and
Leibniz cohomologies are detected. The cases $P$ is a free
$\NP^{r}$ or $\NP^{l}$-algebra, the Hochschild  or/and Leibniz
cohomological dimension of $P$ is $\leq n$ are considered separately,
exhibiting interesting possibilities of representations of the new cohomologies
by the well-known ones and relations between the corresponding cohomological dimensions.
\end{abstract}


\maketitle

\section{Introduction}

   In \cite{CDNo} are defined and studied noncommutative Leibniz-Poisson algebras, denoted as $\NLP$-algebras. These are associative algebras $P$,
   generally noncommutative, over a ring $\K$ with unit, with bracket operation, according to
which they are Leibniz algebras over $\K$ and the Poisson identity
holds
\begin{equation}\label{cd1}
[a \cdot b,c]=a \cdot [b,c]+[a,c]\cdot b
\end{equation}
for all $a,b,c \in P$. In this paper this identity will be called
left Poisson identity and the above defined algebra left
noncommutative Poisson algebra, shortly left $\NP$-algebra or
$\NP^l$-algebra. It is natural to consider right $\NP$-algebras over
a ring $\K$ ($\NP^r$ in what follows), which are defined in
analogous way satisfying the right Poisson identity
\begin{equation}\label{cd2}
[a, b\cdot c]=b\cdot [a,c]+[a,b]\cdot c
\end{equation}
for all $a,b,c \in P$.

A left-right $\NP$-algebra ($\NP^{lr}$) over a ring $\K$ is an
algebra, which is an associative and Leibniz algebra
 and satisfies both \eqref{cd1} and \eqref{cd2} identities; it is a noncommutative analogue
  of the classical Poisson algebra. In the same way, an algebra with bracket $\AWB$ defined
   in \cite{CP}, see  below Definition \ref{D:AWB}, is a left $\AWB$, which will be denoted by $\AWB^l$.  Obviously, we can
    define in analogous ways $\AWB^r$ and $\AWB^{lr}$ as well. Thus we obtain the following commutative diagram
     of the corresponding categories and inclusion functors
\begin{equation} \label{diag}
\xymatrix{
 \mathbf{AWB}^r  & \ \mathbf{AWB}^{lr} \ar@{_{(}->}[l]  \ar@{^{(}->}[r] \ & \mathbf{AWB}^l    \\
  \mathbf{NP}^r \ar@{^{(}->}[u] & \ \ar@{_{(}->}[l]  \mathbf{NP}^{lr}  \ar@{^{(}->}[u] \ar@{^{(}->}[r] \ & \mathbf{NP}^l . \ar@{^{(}->}[u]}
\end{equation}

 The purpose of this paper is to study properties of the above defined algebras, including the construction
 of appropriate complexes for the definition of cohomology, to investigate
 and to establish relations between them and with the properties of the underlying associative and Leibniz algebras and
 the corresponding Hochschild \cite{HoCo}, Quillen \cite{Quillen} and Leibniz cohomologies \cite{LP}. We will see that left-right $\NP$-algebras
 do not inherit all the properties of left or right $\NP$-algebras. But nevertheless they have interesting intersections and relations with each
 other due to the specific way of construction of cohomology complexes.
  We will often omit the proofs, which are analogous to those given in the cases of $\AWB^l$ and $\NP^l$-algebras
 in \cite{CP} and \cite{CDNo}, respectively.

 In Section \ref{S:2} we present the definitions of new algebras considered in
 the paper and examples. For convenience of the reader we include the definition of
category of interest and some examples as well. In Section \ref{S:3} we
construct free $\NP^{lr}$ and $\NP^{r}$-algebras (resp. $\AWB^{lr}$
and $\AWB^r$). The constructions of free $\NP^l$-algebras and
$\AWB^l$ were given in \cite{CDNo} and \cite{CP}, respectively. The
properties of free objects are investigated, in particular, it is
proved that if $P$ is a free $\NP^l$-algebra, then the underlying
associative and Leibniz algebras of $P$ are free as well
(cf. \cite{CDNo}). In Section \ref{S:4} we describe action conditions, we
present definitions of derivation, extension, crossed module and
representation in the categories of the new defined algebras. All these
are special cases of the well-known definitions in categories of
groups with operations. It turned out that the category of
$\NP^{lr}$-algebras is a category of interest, from which, applying
the general result of \cite{Mo}, we conclude that this category is
action accessible in the sense of \cite{BJ}. We construct the
universal strict general actor $\USGA(A)$ of an $\NP^{lr}$-algebra
$A$, defined in \cite{CDLAc} in a category of interest; we describe
center and define actor of $\NP^{lr}$-algebras and, as a special
case of the result in \cite{CDLAc}, we obtain the necessary and
sufficient conditions for the existence of an actor of $A$ in terms
of $\USGA(A)$. We plan to consider the problem of the existence of
an actor in $\mathbf{NP}^{lr}$, or to find individual objects in
this category with actor. According to \cite{BJK} this problem in
categories of interest is equivalent to the amalgamation property
for protosplit monomorphisms. Here in $\mathbf{NP}^{lr}$ we
determine the full subcategory of commutative von Neumann
regular rings with trivial bracket operations; by the result of
\cite{BJK} we have that in this category always exists an actor for
any algebra, and moreover, on the base of the result of the same
paper and \cite {Cor} we  conclude that in $\mathbf{NP}^{lr}$ there
exists a subcategory which satisfies the amalgamation property.
This result can be applied to the characterization of effective
codescent morphisms in this subcategory. In Section \ref{S:5} we
construct complexes and define the corresponding cohomologies
$H_{\NP^{lr}}^n(P,M), H_{\AWB^{lr}}^n(P,M)$, where
$P\in \mathbf{NP}^{lr}$ ($P\in \mathbf{AWB}^{lr}$, respectively),
and $M$ denotes the corresponding representations of $P$.  In what
follows under $\NP$-algebras we will mean $\NP^r$, $\NP^l$ and
$\NP^{lr}$-algebras, and under $\AWB$ we will mean $\AWB^l$,
$\AWB^r$ and $\AWB^{lr}$. We investigate
the relation of the second cohomology with extensions. Like in the
case of $\AWB^l$ \cite {CP}, we obtain the isomorphism
$H_{\AWB^r}^{n+1}(P,M) \approx H_{\mathrm Q}^n(P,M)$ with the
Quillen cohomology. From the
constructions of the cohomology complexes we detect short exact
sequences, from which  follow long exact sequences involving
cohomologies, relating $\NP, \AWB$, Hochschild and Leibniz
cohomologies with each other. The special cases, where $P$ is a free
$\NP^{l}$ or $\NP^{r}$-algebra, the Leibniz cohomological
dimension or/and the Hochschild cohomological dimension of $P$ is
 $n \: / \leq n$ give interesting results, in particular, in these cases we
can represent the new cohomologies by the well-known ones and
estimate cohomological dimensions of the corresponding $\AWB$ and
$\NP$-algebras. Note that an operadic approach to similar kind of
investigations would be interesting, see e. g. \cite{Dot, Fresse1, Hoff, LV}. The cohomology of
classical Poisson algebras is defined and studied by J. Huebschmann
\cite{HuPo}. Different types of noncommutative Poisson algebras were
studied in \cite{Kubo1,Kubo2, Ping, Tong}.

 \section{Preliminary definitions and examples}\label{S:2}

Let $\K$ be a commutative ring with unit. We recall that a {\it
Leibniz algebra} \cite{LoCH,Lo} $A$ over $\K$ is a $\K$-module $A$
equipped with a $\K$-module homomorphism $[-,-] \colon A \otimes A \to A$,
called a square bracket, satisfying the Leibniz identity
\begin{equation}\label{cd3}
[a,[b,c]]=[[a,b],c]-[[a,c],b],
\end{equation}
for all $a,b,c \in A$. Here and in what follows $\tp$ means
$\tp_\K$ .
\begin{definition}  \label{D:AWB} \
\begin{enumerate}
  \item[(i)] A left (resp. right) algebra with bracket over $\K$, for short, $\AWB^l$ (resp. $\AWB^r$), is an associative algebra $A$ equipped with a $\K$-module
homomorphism $[-,-] \colon A \otimes A \to A$, such that \eqref{cd1} (resp.\eqref{cd2}) identity holds.
  \item[(ii)] A left-right algebra with bracket over $\K$ (for short, $\AWB^{lr}$) is an associative algebra $A$ equipped with a $\K$-module
homomorphism $[-,-] \colon A \otimes A \to A$, such that \eqref{cd1} and  \eqref{cd2} identities hold.
\end{enumerate}
\end{definition}

As we have noted in the introduction $\AWB^l$ is the same as algebra
with bracket $\AWB$ defined in \cite{CP} and $\NP^l$-algebra  is
$\NLP$-algebra defined in \cite{CDNo}. Morphisms between the above
defined algebras are $\K$-module homomorphisms preserving the dot
and bracket operations. The corresponding categories will be denoted
by $\mathbf{NP}^l$, $\mathbf{NP}^r$, $\mathbf{NP}^{lr}$,
$\mathbf{AWB}^l$, $\mathbf{AWB}^r$ and $\mathbf{AWB}^{lr}$. The sign ``$\cdot$'' of the dot operation
will be often omitted, when it is clear from the context, which
operation is meant between the elements, e.g. $a \cdot b$ will be
written as $ab$.
\begin{example} \
\begin{enumerate}
  \item[(1)]  Every Poisson algebra is an $\NP^{lr}$-algebra.
  \item[(2)] Any Leibniz algebra $A$ is an $\NP^{lr}$-algebra with trivial
dot operation, i.e. $ab=0, a,b\in A$.
  \item[(3)] Any associative algebra $A$ is an $\NP^{lr}$-algebra with the usual bracket $[a,b]=ab-ba,a,b\in A$.
  \item[(4)]  Let $A$ be an associative algebra and let $D \colon A\rightarrow A$ be a square zero derivation, i.e.  $D^2=0$ and $D(ab)=(Da)b+a(Db)$. Define the bracket operation by $[a,b]=a(Db)-(Db)a$. It is easy to check that with this bracket operation $A$ is an $\NP^{l}$-algebra, but not $\NP^{r}$-algebra.
  \item[(5)] Let $A$ be an associative algebra of the case (4), where the bracket operation is defined by $[a,b]=(Da)b-b(Da)$. Then $A$ is  $\NP^{r}$-algebra, but not $\NP^{l}$-algebra.
  \item[(6)] Let $A$ be an associative algebra with the property that $abc=bac=acb$, for any $a,b$ and $c\in A$, and let $D \colon A\rightarrow A$ be a square zero derivation. Then $A$ is an $\NP^{lr}$-algebra  with respect to the rule  $[a,b]=a(Db)-(Db)a$.
  \item[(7)] Every $\NP$-algebra is an $\AWB$.
  \item[(8)]  The following algebra is an $\AWB^{r}$ (resp. $\AWB^{l}$), but not an $\NP^{r}$-algebra (resp. $\NP^{l}$-algebra). Let $A$ be an associative algebra with a linear application $D \colon A\rightarrow A$. Then $A$ is an $\AWB^{r}$ (resp. $\AWB^{l}$) where the bracket operation is defined by $[a,b]=(Da)b-b(Da)$ \big(resp. by $[a,b]=a(Db)-(Db)a$\big); for the left $\AWB$ this example was given in \cite{CP}.
  \item[(9)]  Let $A$ be an associative algebra with a linear application $D \colon A\rightarrow A$, satisfying the condition $(Da)b-b(Da)= a(Db)-(Db)a$, for any $a,b \in A$. Then the algebra defined in the case (8) is an $\AWB^{lr}$.
   \item[(10)] If the linear application $D \colon A\rightarrow A$ in  the case (9) is a square zero derivation like in case (4), then the algebra with respect to the square bracket  $[a,b]=(Da)b-b(Da)$ is an $\NP^{lr}$-algebra.
  \item[(11)]  Any associative dialgebra \cite{Lo4}
with respect to the operations $ab=a\vdash b$, $[a,b] = a \vdash
b - b \dashv a$ (resp. $[a,b] = a \dashv
b - b \vdash a$) is an $\AWB^{r}$ (resp. $\AWB^{l}$), but not an $\AWB^{l}$ (resp. $\AWB^{r}$).
  \item[(12)] The algebras defined in the case (11) are not generally $\NP^{r}$ and  $\NP^{l}$-algebras, respectively.
  The greatest quotient of these algebras by the congruence relation generated by the relation $[a,[b,c]]\sim[[a,b],c]-[[a,c],b]$, for any $a,b$ and $c\in A$, give examples of $\NP^{r}$ and $\NP^{l}$-algebras, respectively. For $\NP^{l}$-algebras this  example was given in \cite{CDNo}.
  \item[(13)] The algebra defined in the case (11), under the additional condition $a \vdash
b - b \dashv a = a \dashv b - b \vdash a$, for any $a,b \in A$, is an $\NP^{lr}$-algebra.
  \item[(14)] For an example of a graded version of $\NP^{l}$-algebra coming from Physics see \cite{Ka}.
  \item[(15)] See Section \ref{S:3} for the constructions of free $\AWB$ and $\NP$-algebras.
\end{enumerate}
\end{example}
\begin{definition}
Let $P \in \mathbf{NP}^{lr}$. 
A subalgebra of  $P$ is an associative and Leibniz subalgebra of
$P$. A subalgebra $R$ of $P$ is called a two-sided ideal if $a \cdot
r, r\cdot a, [a,r], [r,a] \in R$, for all $a \in P, r \in R$.
\end{definition}
The inclusion functor $\inc \colon \mathbf{Poiss} \to \mathbf{NP}$
from the category of Poisson algebras to the category of
$\NP$-algebras, i.e. left, right or left-right noncommutative
Poisson algebras, respectively,  has a left adjoint $(-)_{\Poiss}
\colon \mathbf{NP} \to \mathbf{Poiss}$. This functor assigns to an
$\NP$-algebra $P$ the quotient algebra of $P$ with the smallest
two-sided ideal spanned by the elements $[x,x]$ and $xy-yx$, for all
$x, y \in P$.

Consider the elements $[a,[b,c\cdot d]], [a,[b\cdot c,d]], [a \cdot b,[c,d]]$ and $[a \cdot b, c\cdot d]$ in the category of $\NP^{lr}$-algebras.
The two different decompositions of the first and the fourth elements give the identities
\begin{equation}
 [a,c]\cdot [b,d]+[a,c]\cdot [d,b]+[b,c]\cdot [a,d]+[c,b]\cdot [a,d]=0
\end{equation}
\begin{equation}
 a \cdot c \cdot [b,d]+[a,c]\cdot d\cdot b=c\cdot a \cdot [b,d]+[a, c] \cdot b \cdot d
\end{equation}

The last identity have place in the category of $\AWB^{lr}$ as well.

The two different decompositions of the second and the third elements do not give identities.

Analogously, considering the two different decompositions of the
first element in the category of $\NP^r$-algebras, and the second
element in the category of $\NP^l$-algebras we obtain, respectively,
the identities
\begin{equation}
 [[a,c]\cdot d, b] = [[a,c],b]\cdot d-[a,c]\cdot [b,d]-[b,c]\cdot [a,d]+c\cdot [[a,d],b]-[c\cdot [a,d],b]
\end{equation}
\begin{equation}
 [a,b\cdot [c,d]] + [a,[b,d]\cdot c]=[[a,b\cdot c],d] - [[a,d],b\cdot c].
\end{equation}

In the categories $\AWB^{lr}$ and $\NP^{lr}$-algebras we have the following identity as well
\begin{equation}
 [a \cdot b,c] - [a,c\cdot b]+[b\cdot c,a] - [b,a \cdot c]+[c\cdot a,b]-[c,b\cdot a]=0.
\end{equation}

By decomposition of the right side of (2.4) we obtain the identity
\begin{equation}
\begin{array}{lcl}
 [[a,c]\cdot d, b] &=& -  [b,[a,c]\cdot d]+[[b,a],c]\cdot d-[[b,c],a]\cdot d-[a,[b,c]\cdot d]+\\
 & & + [[a,b],c]\cdot d +[[a,d],c\cdot b]-[[a,d], c]\cdot b-[c\cdot [a,d],b].
 \end{array}
\end{equation}

These identities will be applied in the next Section. The case of $\NP^l$-algebras was considered in \cite{CDNo}.

Recall that  an action (a derived action in the sense of \cite{Orzech}) of $P$ on $M$ for associative
algebras is given by two $\K$-module homomorphisms $-\cdot - \colon P
\otimes M \to M$, $-\cdot - \colon  M \otimes P \to M$ with the conditions
\begin{align*}
p\cdot(m_1\cdot m_2) =& \ (p\cdot m_1)\cdot m_2 \, ; \ & \  m_1\cdot(p\cdot m_2)= & \ (m_1\cdot p)\cdot m_2 \, ; \\
(m_1\cdot m_2)\cdot p = & \   m_1\cdot(m_2\cdot p)\, ; \ & \  p_1\cdot(p_2\cdot m)=  & \ (p_1\cdot p_2)\cdot m  \, ; \\
p_1\cdot (m\cdot p_2) = & \ (p_1\cdot m)\cdot p_2\, ; \ & \  m\cdot (p_1\cdot p_2)= & \ (m\cdot p_1)\cdot p_2 \, .
\end{align*}
An action of $P$ on $M$ for Leibniz algebras is given by two
$\K$-module homomorphisms $[-,-] \colon P \otimes M \to M$,
$[-,-] \colon M \otimes P \to M$ with the conditions
\begin{align*}
[p,[m_1,m_2]]  & =  [[p,m_1],m_2]-[[p,m_2],m_1]; \\
  [m_1,[p,m_2]] & =  [[m_1,p],m_2]-[[m_1,m_2],p]; \\
[m_1,[m_2,p]] & =  [[m_1,m_2],p]-[[m_1,p],m_2]; \\
   [p_1,[p_2,m]] & =  [[p_1,p_2],m]-[[p_1,m],p_2];\\
 [p_1,[m,p_2]]  & = [[p_1,m],p_2]-[[p_1,p_2],m]; \\
    [m,[p_1,p_2]] & =  [[m,p_1],p_2]-[[m,p_2],p_1].
\end{align*}

Here we recall the definition of category of interest.
Let $\bC$ be a category of groups with a set of operations $\Om$ and
with a set of identities $\bE$, such that $\bE$ includes the group
identities and the following conditions hold. If $\Om_i$ is the set
of $i$-ary operations in $\Om$, then:
\begin{itemize}
\item[(a)] $\Om=\Om_0\cup\Om_1\cup\Om_2$;
\item[(b)] the group operations
\big(written additively: ($0,-,+$)\big) are elements of $\Om_0$, $\Om_1$ and
$\Om_2$ respectively. Let $\Om_2'=\Om_2\setminus\{+\}$,
$\Om_1'=\Om_1\setminus\{-\}$ and assume that if $*\in\Om_2$, then
$\Om_2'$ contains $*^\circ$ defined by $x*^{\circ}y=y*x$. Assume
further that $\Om_0=\{0\}$;
\item[(c)] for each $*\in\Om_2'$, $\bE$
includes the identity $x*(y+z)=x*y+x*z$;
 \item[(d)] for each
$\om\in\Om_1'$ and $*\in\Om_2'$, $\bE$ includes the identities
$\om(x+y)=\om(x)+\om(y)$ and $\om(x)*y=\om(x*y)$.
\end{itemize}
Note that the group operation is denoted additively, but it is not
commutative in general. A category $\bC$ defined above is called a
\emph{category of groups with operations}. The idea of the
definition comes from \cite{Hig} and the axioms are from \cite{Orzech}
and \cite{Porter}. We formulate two more axioms on $\bC$ (Axiom (7)
and Axiom (8) in \cite{Orzech}).

If $C$ is an object of $\bC$ and $x_1,x_2,x_3\in C$:

\begin{itemize}
  \item[] Axiom 1.
 \begin{gather*}
x_1+(x_2*x_3)=(x_2*x_3)+x_1, \,\, \text{for each} \,\, *\in\Om_2'.
\end{gather*}
  \item[] Axiom 2.

  For each ordered pair
$(*,\ol{*})\in\Om_2'\times\Om_2'$ there is a word $W$ such that
\begin{gather*}
(x_1*x_2)\ol{*}x_3=W(x_1(x_2x_3),x_1(x_3x_2),(x_2x_3)x_1,\\
(x_3x_2)x_1,x_2(x_1x_3),x_2(x_3x_1),(x_1x_3)x_2,(x_3x_1)x_2),
\end{gather*}
where each juxtaposition represents an operation in $\Om_2'$.
\end{itemize}

A category of groups with operations satisfying  Axiom 1 and Axiom 2
is called a {\it category of interest} in \cite{Orzech}.

Denote by $\bE_G$ the subset of identities of $\bE$ which includes
the group laws and the identities (c) and (d). We denote by $\bC_G$
the corresponding category of groups with operations. Thus we have
$\bE_G \hookrightarrow \bE$, $\bC=(\Om,\bE)$, $\bC_G=(\Om,\bE_G)$
and there is a full inclusion  functor $\bC\hookrightarrow\bC_G$.
The category $\bC_G$ is called a \emph{general category of groups with
operations} of a category of interest $\bC$ (see \cite{CDL2,CDLAc}).

\begin{example}[Categories of interest] The categories of groups, modules over a ring, vector spaces, associative algebras, associative commutative algebras, Lie algebras, Leibniz algebras are categories of interest. In the example of groups
$\Om_2'=\varnothing$. In the case of associative algebras with
multiplication represented by  $*$, we have $\Om_2'=\{*,*^\circ\}$.
For Lie algebras take $\Om_2'=\{[\;,\;],[\;,\;]^\circ \}$ (where
$[a,b]^\circ=[b,a]=-[a,b]$). For Leibniz algebras, take
$\Om_2'=\{[\;,\;],[\;,\;]^\circ \}$, (here $[a,b]^\circ=[b,a]$). The
category of  alternative algebras is a category of interest as well
\cite{Orzech} (see also \cite{CDL4}). The categories of crossed modules  and
precrossed modules in the category of groups, respectively, are
equivalent to categories of interests (see e.g. \cite{CDLAc,CDL3}).
 According to \cite{BJK} the category of commutative Von
Neumann regular rings is isomorphic to a category of interest. In
\cite{Mo} are given new examples of categories of interest, these
are associative dialgebras and associative trialgebras. Dialgebras
and trialgebras were defined by Loday \cite{Lo3,Lo4,LoRo}. As it is
noted in \cite{Orzech} Jordan algebras do not satisfy Axiom 2. It is
easy to see that $\mathbf{NP}^{lr}$  is a category of interest;
while the categories $\mathbf{AWB}^{lr}$, $\mathbf{AWB}^r$,
$\mathbf{AWB}^l$, $\mathbf{NP}^r$ and $\mathbf{NP}^l$  are not
categories of interest, they do not satisfy Axiom 2 of the
definition.
\end{example}

\section{Free objects in $\mathbf{NP}$ and $\mathbf{AWB}$} \label{S:3}
For any set $X$ we shall build a free $\NP^{lr}$-algebra (resp.
$\NP^r$) over a ring $\K$. The construction for $\NP^{l}$-algebras
is given in \cite{CDNo}. Denote by $W(X)$  \big(resp. $W'(X)$\big)
the set, which contains $X$ and all formal combinations (words) of
two operations $(\cdot,[-,-])$ with the elements from $X$, which
have a sense, and which do not contain elements of the forms
$\big[a,[b,c]\big]$, $[a \cdot b,c]$, $[a,b\cdot c]$ (resp.
$\big[a,[b,c]\big], [a,b\cdot c]$ and $[[a,c]\cdot d,b]$); moreover, if the element
$a \cdot b\cdot [c,d]$ is contained in any word, then the element
$b\cdot a \cdot [c,d]$ is not contained in any word, and if the
element $[a,b]\cdot [c,d]$ is contained in any word, then
$[a,b]\cdot [d,c]$ is not contained in any word, where $a,b,c,d$ are from $X$ or are combinations of
elements of $X$ and dot and bracket operations. Let $W_n(X)$
\big(resp. $W'_n(X)$\big) be the subset of those words of $W(X)$ (resp. $W'(X)$),
which contain $n$ elements of $X$, i.e. the number of both
operations together is $n-1$, $n\geq 1$; we say that this word is of
length $n$. Obviously, $W(X)= \bigcup_{n \geq 1} W_n(X) $ (resp.
$W'(X)= \bigcup_{n \geq 1} W'_n(X) $). We define the following maps
\[\sigma_{n,m}, \tau_{n,m} \colon W_n(X) \times W_m(X) \longrightarrow
W_{n+m}(X) \, . \] $\sigma_{n,m}$ is defined only on those pairs $(a,b)\in W_n(X) \times W_m(X)$, for
which the word $a \cdot b \in W_{n+m}(X)$, and by definition
 $\sigma_{n,m}(a,b) = a \cdot b$, where the
right side denotes the word from $W_{n+m}(X)$, which is defined
uniquely. Analogously, $\tau_{n,m}$ is defined only on those pairs $(a,b)$, for
which the word $[a,b] \in W_{n+m}(X)$, and by definition
$\tau_{n,m}(a,b)=[a,b]$. In the case $a \cdot b, [a,b] \notin W_{n+m}(X)$,
 $\sigma_{n,m}$ and $\tau_{n,m}$  are not defined. $\sigma'_{n,m}$ and $\tau'_{n,m}$ are
defined in analogous ways; it is easy to notice, that $\sigma'_{n,m}$ is defined for any pair $(a,b)\in W'_n(X) \times W'_m(X)$ by $\sigma'_{n,m}(a,b)=a \cdot b$. Let $F\big(W(X)\big)$ \big(resp. $F\big(W'(X)\big)$\big) be the
free $\K$-module generated by the set $W(X)$ \big(resp. $W'(X)$\big). Define
the dot operation on $F\big(W'(X)\big)$ as a linear
extension of $\sigma'_{n,m}$ on whole $F\big(W'(X)\big).$ For those words of  $F\big(W(X)\big)$ (resp. $F\big(W'(X)\big)$) on which
$\sigma_{n,m}$ and $\tau_{n,m}$ are defined (resp. $\tau'_{n,m}$ is defined), we define the dot and the bracket operations (resp. bracket operation) as
the linear extensions on  $F\big(W(X)\big)$ of $\sigma_{n,m}$ and $\tau_{n,m}$, respectively (resp. on $F\big(W'(X)\big)$
of $\tau'_{n,m}$). If the elements $a \cdot b,[a,b] \notin W_{n+m}(X)$, for $ a \in
W_n(X), b \in W_m(X)$ \big(resp. $[a,b] \notin W'_{n+m}(X)$, for $ a \in
W'_n(X), \ b \in W'_m(X)$\big), we decompose $a \cdot b$ and $[a,b]$ (resp. $[a,b]$) according to the
identities \eqref{cd1}, \eqref{cd2}, \eqref{cd3}, (2.2), (2.3) and the $\mathbb{K}$-linearity of the dot and the bracket operations \big(resp. \eqref{cd2},
\eqref{cd3}, (2.4) and the $\mathbb{K}$-linearity of  bracket operation\big), until
 we obtain the sum of the dot and the bracket operations, respectively (resp. the sum of the bracket operations), on such pairs of words on which $\sigma_{n,m}$ and $\tau_{n,m}$  are defined (resp. on which $\tau'_{n,m}$ is defined).
Acting on every step in such  ways  we will obtain the sums $d_1+
\dots+d_l$ and $c_1+
\dots+c_k$, respectively, (resp. $c'_1+
\dots+c'_t$), with  $b_j, c_i \in F\big(W_{n+m}(X)\big)$ (resp. $c'_i\in F\big(W'_{n+m}(X)\big)$) and by definition
$(a,b)=d_1+ \dots+d_l$ and $[a,b]=c_1+ \dots+c_k$ (resp. $[a,b]=c'_1+ \dots+c'_t$).  Any two
different decompositions give the same element of $F\big(W(X)\big)$
\big(resp. $F\big(W'\big(X)\big)$\big) and the results of the
operations are uniquely defined. By construction $F\big(W(X)\big)$
\big(resp. $F\big(W'(X)\big)$\big) has a structure of
$\NP^{lr}$-algebra (resp. $\NP^r$-algebra). Let $i \colon
X\longrightarrow F\big(W(X)\big)$ be the natural injection of sets.
\begin{proposition}
For any $\NP^{lr}$-algebra $B$ and a map ${\varphi} \colon X \to B$, there
exists a unique homomorphism $\bar{\varphi} \colon F\big(W(X)\big)
\to B$ such that the following diagram commutes
\[
\xymatrix{
X \ar[r]^-i \ar[d]_-\vf & F\big(W(X)\big)  \ar[dl]^-{\bar{\varphi}} \\
B
}
\]
\end{proposition}
We omit the proof since it is analogous to the one of
\cite[Proposition 2.2.1]{CDNo} for $\NP^l$-algebras. We have for the
$\NP^r$-algebra $F\big(W'(X)\big)$ the analogous proposition. The
constructions of free $\AWB^l$, $\AWB^r$ and $\AWB^{lr}$ are similar
to the construction given above; e.g. in the case of free $\AWB^l$
we take all formal combinations (words) of two operations
$(\cdot,[-,-])$ with the elements from $X$, which have a sense, and
do not contain the elements of the form  $[a \cdot b,c]$ (cf. with
the construction given in \cite{CP}).

It is easy to see that the given construction defines
the functor $\mathbf{F}$ from the category of sets $\mathbf{Set}$ to
$\mathbf{NP}^{lr}$, where $\mathbf{F}(X)=F\big(W(X)\big)$, which is a left
adjoint to the underlying functor
\[
\xymatrix@C=36pt{\mathbf{Set}
\ar@<0.7ex>[r]^-{\mathbf{F}}& \ar@<0.7ex>[l]^-{\mathbf{U}} \mathbf{NP}^{lr}. }
\]
 We have between the category
$\mathbf{Set}$ and the categories $\mathbf{NP}^l$ \cite{CDNo},
$\mathbf{NP}^r$, $\mathbf{AWB}^{lr}$, $\mathbf{AWB}^l$ and
$\mathbf{AWB}^r$ the analogous pairs of adjoint functors. Let $V^{lr}_{A} \colon \mathbf{NP}^{lr}\rightarrow \mathbf{Ass}$,
 $V^{lr}_{L} \colon \mathbf{NP}^{lr}\rightarrow \mathbf{Leib}$ and $T^{lr}_{A} \colon \mathbf{AWB}^{lr}\rightarrow \mathbf{Ass}$  be the forgetful functors, where $ \mathbf{Ass}$ and $ \mathbf{Leib}$ denote the categories of associative and Leibniz algebras, respectively. The analogous meaning will have the symbols
$V^r_{A}$, $V^l_{A}$, $V^r_{L}$, $V^l_{L}$, $T^r_{A}$, $T^l_{A}$.
\begin{proposition} \label{free_aL}
If $P$ is a free $\NP^l$-algebra, then $V^l_{A}(P)$ and
$V^l_{L}(P)$ are free associative and free Leibniz algebras,
respectively.
\end{proposition}
\begin{proof} This theorem in the terminology of $\NLP$-algebras is proved in \cite{CDNo}, but the proof for the  Leibniz algebras case needs a correction. Therefore we omit the proof of the associative algebra case and present the proof of the second part of the proposition.
Let $P$ be the free $\NP^l$-algebra on the set $X$. Let $X''$ be the
set of all kind of words of the types $a_{1} \cdot \ldots  \cdot
a_{n}$ and $a_{1} \cdot \ldots  \cdot a_{n} \cdot [\dots,\dots]$,
where $a_{1},\dots ,a_{n}\in X$, $n\geq1$ and the bracket
$[\dots,\dots]$ does not contain the words of the forms $[a,[b,c]]$
and $[a,[b,d]\cdot c]$, for $a,b,c,d \in P$. Let $X_{2}=X\bigcup
X''$. Applying (1.1), (2.1) and (2.5) it is easy to show that
$V^l_{L}(P)$ is the free Leibniz algebra on the set $X_{2}$.
\end{proof}

Below we will see that the proof of the analogous statement for $\NP^r$-algebras is more complicated.

\begin{proposition}  \label{P:free_npr}
If $P$ is a free $\NP^r$-algebra (resp. $\AWB^r$), then $V^r_{A}(P)$ and
$V^r_{L}(P)$ (resp. $T^r_{A}(P)$)  are free associative and free Leibniz algebras,
respectively (resp. free associative algebra).
\end{proposition}
\begin{proof}Let $P$ be the free $\NP^r$-algebra (resp. $\AWB^r$) on the set $X$. Denote by $X'$ the set of all kind of those
words of the type $[...,...],$ which doesn't contain the words of the forms $ [a,[b,c]], [a,bc],$ and [[a,c]d,b] (resp. $[a,bc]$). Let $X_1 = X\bigcup X'$. It is easy to check that $V^r_{A}(P)$ (resp. $T^r_{A}(P)$) is a free associative algebra on the set $X_1.$

Let $X''$ be the set of all kind of words of the types $a_{1} \cdot
\ldots  \cdot a_{n}$ and $[\dots,\dots]\cdot a_{1} \cdot \ldots
\cdot a_{n},$ where $a_{1},\dots,a_{n}\in X$, $n\geq1$ and the
bracket $[\dots,\dots]$ does not contain the words of the form
$[a,[b,c]],$ and moreover, for any $a,b,c,d\in P$, the fixed words
$[[a,b],c]\cdot d, [[a,c],b]\cdot d, [[a,b],d]\cdot c, [[d,a],b]\cdot
c, [[d,b],a]\cdot c, [[c,a],d]\cdot b, [[c,d],a]\cdot b$ and
$[[c,d],b]\cdot a$ do not belong to $X''$. All other 16 combinations
of the type $[[x,y],z]\cdot t$ of the elements $a,b,c,d$ are in
$X''$. Applying identities (1.2), (2.1) and (2.7) it is easy to
check that $V^r_{L}(P)$ is a free Leibniz algebra on the set
$X_2=X\bigcup X''.$
\end{proof}
An analogous statement for $\AWB^l$ is proved in \cite{CP}.
The following statement proves that left-right $\NP$-algebras do not inherit all properties of left or right $\NP$-algebras.
\begin{proposition} \label{P:free_nplr}
If $P$ is a free $\NP^{lr}$-algebra (resp. $\AWB^{lr}$), then $V^{lr}_{A}(P)$ (resp. $T^{lr}_{A}(P)$)  is not a free associative  algebra and $V^{lr}_{L}(P)$ is not a free Leibniz algebra.
\end{proposition}
\begin{proof} Let $P$ be the free $\NP^{lr}$-algebra on the set $X$.
A basis for $V^{lr}_{A}(P)$ must contain all the elements from $X$,
and all the elements of the form $[a,b]$, where $a,b\in X$. From
identity (2.2) or (2.3) it follows that $V^{lr}_{A}(P)$ is not a
free associative algebra. Analogously, from identity (2.3) we see
that $T^{lr}_{A}(P)$ is not a free associative  algebra. In the case
of the Leibniz algebra $V^{lr}_{L}(P)$, its basis must contain all
the elements from $X$ and all kind of elements of the form $a_{1}
\cdot \ldots \cdot a_{n}$, where $a_{1},\dots,a_{n}\in X$, $n\geq1$.
The identity (2.6) proves that $V^{lr}_{L}(P)$ is not a free Leibniz
algebra.
\end{proof}

\section{Actions, representations and crossed modules in $\mathbf{NP}$ and $\mathbf{AWB}$} \label{S:4}
Under action we will mean a set of actions derived from the
corresponding split extension, i.e. a derived action in the sense of
\cite{Orzech}. An action for $\NP^l$-algebras is defined in  \cite{CDNo} in the following way:
\begin{definition}[\cite{CDNo}]
Let $M, P \in \mathbf{NP}^l$. We say that  $P$ acts on $M$ if we
have an action of $P$ on $M$ as associative and Leibniz algebras
given respectively by the $\K$-module homomorphisms
\begin{equation}\label{E:dot}
 -\cdot - \colon P \otimes M
\to M \, , \quad \quad \quad -\cdot- \colon M \otimes P \to
M\,\end{equation}
\begin{equation} \label{E:bracket}
[-,-] \colon P \otimes M \to M  \, ,
\quad \quad \quad [-,-] \colon M \otimes P \to M
\end{equation}
 and the following conditions hold:
\begin{align*}
[p_1\cdot p_2,m]  &  = \, p_1\cdot  [p_2,m]+[p_1,m]\cdot  p_2 ; \ &
   [p_1\cdot  m,p_2]  & =  \, p_1\cdot  [m,p_2]+[p_1,p_2]\cdot  m ;\\
[m\cdot  p_1,p_2]  & =  \, m\cdot  [p_1,p_2]+[m,p_2]\cdot  p_1 ; \ &
    [m_1\cdot  m_2,p]  & =  \, m_1\cdot  [m_2,p]+[m_1,p]\cdot  m_2;\\
[m_1\cdot  p,m_2] &  = \, m_1\cdot  [p,m_2]+[m_1,m_2]\cdot  p; \ &
   [p\cdot  m_1,m_2] & =  \, p\cdot  [m_1,m_2]+[p,m_2]\cdot  m_1,
\end{align*}
for all $m,m_1,m_2 \in M; p,p_1,p_2 \in P$.
\end{definition}
\begin{definition}
Let $M, P \in \mathbf{NP}^r$. We say that  $P$ acts on $M$ if we have an
action of $P$ on $M$ as associative and Leibniz algebras  given
 by the $\K$-module homomorphisms \eqref{E:dot} and \eqref{E:bracket}, respectively,
 and the following conditions hold
\begin{align*}
[m,p_1\cdot p_2]  &  = \, p_1\cdot [m,p_2]+[m,p_1]\cdot p_2 ; \ &
 [p_1,p_2\cdot m]  & = \, p_2\cdot[p_1,m]+[p_1,p_2]\cdot m; \\
[p_1,m\cdot p_2]  & = \, m \cdot [p_1,p_2]+[p_1,m]\cdot p_2; \ &
 [p,m_1\cdot m_2]  & = \, m_1\cdot[p,m_2]+[p,m_1]\cdot m_2;\\
[m_1,m_2\cdot p] & =  \, m_2\cdot [m_1,p]+[m_1,m_2]\cdot p; \ &
[m_1,p\cdot m_2]  & = \, p\cdot[m_1,m_2]+[m_1,p]\cdot m_2,
\end{align*}
for all $m,m_1,m_2 \in M$ and $p,p_1,p_2 \in P$.
\end{definition}
\begin{definition} \label{D:act_nplr}
Let $M, P \in \mathbf{NP}^{lr}$. We say that  $P$ acts on $M$ if we
have an action of $P$ on $M$ as left and right $\NP$-algebras.
\end{definition}

Actions in the categories $\mathbf{AWB^l}$, $\mathbf{AWB^r}$ and
$\mathbf{AWB^{lr}}$ are defined in similar ways as in the previous
definitions, but obviously, the Leibniz algebra action conditions
are not required. If an $\NP$-algebra $P$ acts on $M$, and $M$ is
singular, or  equivalently abelian, i.e. $M\cdot M=[M,M]=0$, then
$M$ will be called  a {\it representation} of $P$. Representation in
the category $\mathbf{AWB}$ (for $\mathbf{AWB^l}$ see \cite{CP}) is
defined in a similar way. These definitions coincide with the
special cases of the general definition of module given in
categories of groups with operations in \cite{Orzech}. If $M$ is a
representation of $P$ in $\mathbf{NP}$, then $M$ is a
$P$-$P$-bimodule, $P$ considered as the underlying associative
algebra; analogously, $M$ is an $\AWB$ representation of $P$ and a
Leibniz representation of $P$ defined in \cite{LP}. In the case of
Poisson algebras we obtain the representation defined in
\cite{Fresse}.

A homomorphism between two representations over $P$ is a linear map
$f \colon M \to M'$ satisfying
 \begin{align*}
 f(p\cdot m)= & \ p\cdot f(m) \, ,   & \  f(m\cdot p)= & \ f(m)\cdot p \, , \\
f[p,m]= & \ [p,f(m)] \, ,    & \ f[m,p]= & \ [f(m),p] \, ,
\end{align*}
 for all $ p \in P$ and $m \in M$.

\begin{definition}
Let $P \in \mathbf{NP}$  and $M$ be a  representation of $P$. A
derivation from $P$ to $M$ is a linear map $d \colon P \to M$ satisfying
the conditions
\begin{align*}d(p_1\cdot p_2)&=d(p_1)\cdot p_2+p_1 \cdot d(p_2)\,,\\
d[p_1,p_2]&=[d(p_1),p_2]+[p_1,d(p_2)] \,.
\end{align*}
\end{definition}
We have the analogous definition for $\AWB$.

Denote by $\Der_{\NP}(P,M)$ the $\K$-module of such
derivations;  analogously we will have the notation
$\Der_{\AWB}(P,M)$. Any $\NP$-algebra $P$ is a
representation of $P$ acting on itself by the operations in $P$
(see  \cite[Example 2.3.2]{CDNo}). For $p \in P$, the application
$\ad_p \colon P \to P$ defined by $\ad_p(p') = - [p',p]$ is an example of
derivation. The following definition is a special case of the
definitions given in \cite{Orzech,Porter}.
\begin{definition}
Let $P, M \in \mathbf{NP}$. An abelian extension of $P$  by $M$ is a
short exact sequence
\[E \colon 0 \to M  \xrightarrow{i}  Q \xrightarrow{j} P \to 0\]
where $Q \in \mathbf{NP}$ and $M$ is abelian.
\end{definition}
Any abelian extension defines on $M$ a unique representation of $P$ in such a way that
\begin{align*}
i(j(q)\cdot m)=q \cdot i(m); \ &  \qquad i(m\cdot j(q)) = i(m)\cdot q ;\\
i[j(q),m]=[q,i(m)]; \ &  \qquad i([m,j(q)]) = [i(m),q] ;
\end{align*}
for any $m \in M, q \in Q$. Two abelian extensions $E$ and $E'$ of
$P$ by $M$ are called \emph{equivalent} if there exists a homomorphism of
$\NP$-algebras $f \colon Q \to Q'$ inducing the identity morphisms on $M$
and $P$. Note that in this case $f$ is an isomorphism. Let $M$ be
any representation of $P$. Denote by $\Ext_{\NP}(P,M)$ the set of all
equivalence classes of those abelian extensions of $P$ by $M$, which
induce the given representation $M$ of $P$.
\begin{definition}
Let $M, P \in \mathbf{NP}$ with an action of $P$ on $M$. A crossed
module is a morphism $\mu \colon M \to P$ in $\mathbf{NP}$ satisfying the
following axioms:
\begin{align*}
\mu(p\cdot m) = & \ p\cdot \mu(m)\, ,  & \mu(m\cdot p) = &\ \mu(m)\cdot p \, , \\
\mu[p,m] = &\ [p,\mu(m)]\, ,  &  \mu[m,p] \ = &  \ [\mu(m),p] \, ,\\
\mu(m)\cdot m' = & \  m\cdot m' \ = \    m \cdot \mu(m') \, ,  &
[\mu(m),m'] =   & \ [m,m']  =  \   [m,\mu(m')] \, .
\end{align*}

A homomorphism of crossed modules is a pair $(\phi,\psi) \colon
(M,P,\mu) \to (M',P',\mu')$ where $\phi, \psi$ are morphisms in
$\mathbf{NP}$ such that $\psi \mu= \mu' \phi$ and $\phi(p\cdot m)=
\psi(p)\cdot \phi(m)$; $\phi(m \cdot p) = \phi(m)\cdot \psi(p); \phi[p,m]
= [\psi(p),\phi(m)]; \phi[m,p] = [\phi(m),\psi(p)]$, for
all $p \in P, m \in M$.
\end{definition}
Examples of representations and crossed modules and the construction
of semi-direct products in the category of $\NP$-algebras and
$\AWB$ are analogous to those given for $\NP^l$-algebras and
$\AWB^l$, therefore for these subjects we refer the
reader to \cite{CDNo} and \cite{CP}, respectively.

It is proved in \cite{Mo} that every category of interest is action
accessible in the sense of \cite{BJ}. Since $\mathbf{NP}^{lr}$ is a
category of interest (see Section \ref{S:2}) we obtain
\begin{theorem}
The category $\mathbf{NP}^{lr}$  is action accessible.
\end{theorem}
In \cite{CDLAc} for any category of interest $\mathbf{C}$ and for
any object $A\in \mathbf{C}$ is defined and constructed the
universal strict general actor $\USGA(A)$ of A, which is generally
an object of $\bC_G$. Here we give this construction for the
category $\mathbf{NP}^{lr}$. In this case we have three binary
operations: the addition, denoted by `` + '', the dot and the
(square) bracket operations. $\Om_2'$ from the definition of
category of interest is a set with three elements $\Om_2'=\{\cdot, [\;,\,], [\;,\,]^\circ
\}$. Since the addition is commutative, the action corresponding to
this operation is trivial. Thus we will deal only with actions,
which are defined by dot and bracket operations; the actions of $b$
on $a$ will be denoted as $a \cdot b, b\cdot a, [b,a]$ and $[a,b]$.
Below under $*$ operation we will mean either dot or bracket
operations. Let $A\in\mathbf{NP}^{lr}$; consider all split
extensions of $A$
\[ \xymatrix{E_j \colon 0\ar[r]&A\ar[r]^-{i_j}&C_j\ar[r]^-{p_j}&B_j\ar[r]&0},\quad j\in \bJ.\]
Let $\{b_j* \mid b_j\in B_j,\;*\in\Om_2'\}$ be the corresponding set of
derived actions for $j\in\bJ$. For any element $b_j\in B_j$ denote
${\bbb}_j=\{b_j*,\;*\in\Om_2'\}$. Let $\bB=\{{\bbb}_j \mid b_j\in
B_j,\;j\in\bJ\}$. Thus each element ${\bbb}_j\in\bB$, $j\in\bJ$, is the
special type of a function ${\bbb}_j \colon \Om_2'\lra\text{Maps}(A\to A)$,
${\bbb}_j(*)=b_j*- \colon A\lra A$. According to Axiom $2$ of the
definition of a category of interest, we define $*$ operation,
${\bbb}_i*{\bbb}_k,*\in\Om_2'$, for the elements of $\bB$ by the
equalities
\begin{align*}
&({\bbb}_i*{\bbb}_k)\ol{*}\,(a)=W(b_i,b_k;a;*,\ol{*}).
\end{align*}
We define
\begin{align*}
({\bbb}_i+{\bbb}_k)*(a)=& \ b_i*a+b_k*a,\\
(-{\bbb}_k)*(a)=&-(b_k * a),\\
(-b)*(a)=&-\big(b*(a)\big),\\
- (b_1+\cdots+b_n)=&-b_n-\cdots-b_1,
\end{align*}
where $*\in \Om_2'$, $b,b_1,\dots,b_n$ are certain combinations of
the dot and the bracket operations on the elements of $\bB$, i.e. the
elements of the type $\bbb_{i_1}*_1\cdots*_{n-1}\bbb_{i_n}$, where
$n>1$. We do not know if
the new functions defined by us are again in $\bB$. Denote by
$\cB(A)$ the set of functions \big($\Om_2'\lra\text{Maps}(A\to A)$\big)
obtained by performing all kinds of the above defined operations on
elements of $\bB$ and the new obtained elements as results of
operations. Let $b\sim b'$ in $\cB(A)$ if $b*a=b'*a$,  for
any $a\in A$, $*\in\Om_2'$. It is an equivalence relation; denote by
$\USGA(A)$ be the corresponding quotient algebra. Let
$\mathbf{NP}^{lr}_G$ be a general category of groups with operations
of the category of interest $\mathbf{NP}^{lr}$.
\begin{proposition}
$\USGA(A)$ is an object in $\mathbf{NP}^{lr}_G$.
\end{proposition}
\begin{proof}
Direct easy checking of the identities.
\end{proof}
As above, we will write for simplicity $b*(a)$ instead of
$\big(b(*)\big)(a)$, for $b\in \USGA(A)$ and $a\in A$. Define a set of
actions of $\USGA(A)$ on $A$ in the following natural way. For $b\in
\USGA(A)$ we define $b*a=b*(a)$, $*\in\Om_2'$. Thus if
$b=\bbb_{i_1}*_1\cdots*_{n-1}\bbb_{i_n}$, where we mean certain
round brackets, we have
\begin{align*}
&b\,\ol{*}\,a=(\bbb_{i_1}*_1\cdots*_{n-1}\bbb_{i_n})\,\ol{*}\,(a).
\end{align*}
The right side of the equality is defined inductively according to
Axiom $2$. For $b_k\in B_k$, $k\in\bJ$, we have
\begin{align*}
\bbb_k*a=\bbb_k*(a)=b_k*a.
\end{align*}
Also
\begin{align*}
&(b_1+b_2+\dots+b_n)*a=b_1*(a)+\dots+b_n*(a).
\end{align*}
\begin{proposition}
The set of actions of  $\USGA(A)$ on $A$ is an action in the
category $\mathbf{NP}^{lr}_G$.
\end{proposition}
\begin{proof} It is a special case of the proof of the general statement for categories of interest given in \cite{CDLAc}. The checking shows that the set of  actions of $\USGA(A)$
on $A$ satisfies conditions of \cite[Proposition 1.1]{DaCtr}, which
proves that it is an action in $\mathbf{NP}^{lr}_G$.
\end{proof}
Note that this is an action in $\mathbf{NP}^{lr}_G$, which in general doesn't satisfy the  action conditions in $\mathbf{NP}^{lr}$.
Define a map $d \colon A\lra \USGA(A)$ by $d(a)=\aaa$, where
$\aaa=\{a\cdot,a*,*\in \Om_2'\}$. Thus we have by definition
\begin{align*}
&d(a)*a'=a*a',\quad \forall a,a'\in A,\quad *\in\Om_2'.
\end{align*}
The proofs of the following two statements are special cases of those
given in \cite{CDLAc}.
\begin{lemma}
$d$ is a homomorphism in $\mathbf{NP}^{lr}_G$.
\end{lemma}
\begin{proposition}
$d \colon A\lra \USGA(A)$ is a crossed module in $\mathbf{NP}^{lr}_G$.
\end{proposition}

According to the general definition of center \cite{Orzech} (cf. with
the definition in \cite{CDLAc}) we describe the center of an object
in $\mathbf{NP}^{lr}$ as follows

\begin{definition} The center of $P\in \mathbf{NP}^{lr}$ is

\[Z(P)=\{z\in P \mid z\cdot p=p\cdot z=[z,p]=[p,z]=0,\ p\in P\}.\]

\end{definition}
It is easy to see that $Z(P)=\Ker d$.

 Here we give the definition of an actor in $\mathbf{NP}^{lr}$ (for the case of a
category of interest see \cite{CDLAc}).
\begin{definition}
For any object $A$ in $\mathbf{NP}^{lr}$ an actor of $A$ is an
object $\rm{Act}(A)\in \mathbf{NP}^{lr}$, which has an action on $A$
in the same category (i.e. satisfying the conditions of Definition
\ref{D:act_nplr}), such that for any object $C$ in $\mathbf{NP}^{lr}$ with an
action on $A$, there is a unique morphism $\vf \colon C\lra \text{Act}(A)$
with
\begin{align*} c\cdot a= & \ \vf(c)\cdot a,   &   a \cdot c & \ =    a \cdot \vf(c), \\
[c,a]= & \ [\vf(c),a],     &  [a,c] & \ =   [a,\vf(c)],
\end{align*}
for any $a\in A$ and $c\in C$.
\end{definition}

 According to the same paper, an actor of $A$ is a split extension classifier for $A$
 in the sense of \cite{BJK1}.
From the results of \cite{CDLAc} we obtain.
\begin{theorem}
For any element $A\in \mathbf{NP}^{lr}$ there exists an actor of $A$ if and only if the semidirect product $\USGA(A)\ltimes A\in \mathbf{NP}^{lr}$. If it is the case, then  $\Actor(A)= \USGA(A)$.
\end{theorem}

At the end of this section we give an example of a subcategory in
$\mathbf{NP}^{lr}$, which satisfies the amalgamation property. This
result can be applied to the description of effective codescent
morphisms in the corresponding subcategory. For the definition of
amalgamation property one can see \cite{BJK}.

Recall that a ring $R$ (generally without a unit) is von Neumann
regular if for any $r\in R$ there exists an element $r'\in R$ such
that $rr'r=r$.

\begin{proposition} In the category of $\NP^{lr}$-algebras there exists a subcategory, which satisfies the amalgamation property.
\end{proposition}
\begin{proof} Consider the full subcategory in
$\mathbf{NP}^{lr}$, whose objects are
commutative von Neumann regular rings with trivial bracket
operations. Now it remains to apply the result of \cite{BJK}, where it is proved that the category of (not necessarily
unital) commutative von Neumann regular rings satisfies the
amalgamation property.
\end{proof}

\section{Cohomology} \label{S:5}

We recall the constructions of complexes for Hochschild and Leibniz
cohomologies,  for cohomologies of left algebras with bracket and
left $\NP$-algebras, i.e. $\AWB^l$ and $\NP^l$-algebras according to
\cite{CP,CDNo}, respectively. Below for $P\in \NP$ instead of
underlying associative and Leibniz algebras $V_A(P), V_L(P)$ and
underlying $\AWB$ we will write for simplicity just $P$ and will note
what kind of algebras we mean, similarly for $P\in \AWB$ and
$T_A(P), T_L(P)$.

Let $P$ be a left $\NP$-algebra over a field $\K$ and $M$  a
 representation of $P$. In particular, $P$ is an associative algebra and
$M$ is a $P$-$P$-bimodule and, on the other hand, $P$ is a Leibniz
algebra and $M$ is a representation of $P$ in the category of
Leibniz algebras. Let $(C^*_H(P,M),\partial^n_H)$ be the Hochschild
complex and $(C^*_L(P,M),\partial^n_L)$ be the Leibniz complex. We
recall that for $n \geq 0$
\[C^n_H(P,M) = C^n_L(P,M) = \Hom(P^{\otimes n},M)\]
 and coboundary maps $\partial^n_H$ and $\partial^n_L$ are
given by
\begin{multline*}
\partial^n_H (f)(p_1,\dots,p_{n+1}) =(-1)^{n+1}\Big\{p_1f(p_2,\dots,p_{n+1})\\ +
\sum_{i=1}^n(-1)^if(p_1,\dots,p_ip_{i+1},\dots,p_{n+1})+(-1)^{n+1}f(p_1,\dots,p_n)p_{n+1}\Big\}\, ,
\end{multline*}
\begin{multline*}\partial^n_L (f)(p_1,\dots,p_{n+1}) =[p_1,f(p_2,\dots,p_{n+1})] \\ + \sum_{i=2}^{n+1}(-1)^i[f(p_1,\dots,\widehat{{p_i}},\dots,p_{n+1}),p_i] \\+
\sum_{1 \leq i < j \leq
n+1}(-1)^{j+1}f(p_1,\dots,p_{i-1},[p_i,p_j],p_{i+1},
\dots,\widehat{p_j},\dots,p_{n+1}) \; .
\end{multline*}
Thus $C^n_H(P,M)$ and  $C^n_L(P,M)$ are complexes of $\K$-vector
spaces. We will need below the $P$-$P$-bimodule $M^e$, defined by
$M^e = \Hom(P, M)$ as a $\K$-vector space, and a bimodule structure on $M^e$
given by $(p_1\cdot f)(p_2) = p_1\cdot f(p_2)$; $(f\cdot
p_1)(p_2)=f(p_2)\cdot p_1$. We have an isomorphism of $\K$-vector
spaces $\theta_n \colon  C^{n+1}_H(P,M) \to C^n_H(P,M^e), n \geq 1$,
defined in an obvious way
$\theta_n(f)(p_1,\dots,p_n)(p)=f(p_1,\dots,p_n,p)$. Denote the
coboundary maps of the complex $C^*_H(P,M^e)$ by $\partial^{e,*}_H$.
Let
\begin{align*}
 \bar C^*_H(P,M)= &  \ (C^n_H(P,M),\partial^n_H,n \geq 1), \\
 \bar C^*_H(P,M^e)= &  \ (C^n_H(P,M^e),\partial^{e,n}_H, n \geq 1),\\
 \bar{C}^*_L(P,M) =  &  \ (C^n_L(P,M),\partial^n_L,n \geq 1).
\end{align*}

Consider the following homomorphisms of cochain complexes, defined in
\cite{CP} and \cite{CDNo}, respectively,
\[ \alpha^*  \colon \bar{C}^*_H(P,M)  \to \bar
C^*_H(P,M^e)\] and
\[\beta^*  \colon \bar{C}^*_L(P,M) \to
\bar{C}^*_H(P,M^e)\]
and given by
\[\alpha^1(f)(p_1)(p_2) = [p_1,f(p_2)]+[f(p_1),p_2]-f([p_1,p_2]),\]
and for $n>1$
\begin{multline*}
\alpha^n(f)(p_1,\dots,p_n)(p_{n+1})=[f(p_1,\dots,p_n),p_{n+1}] -
f([p_1,p_{n+1}],p_2,\dots,p_n) \\ - f(p_1,[p_2,p_{n+1}],\dots,p_n) - \dots -
f(p_1,\dots,p_{n-1},[p_n,p_{n+1}]),
\end{multline*}
\begin{align*}
\beta^{2k+1} = &  \ \theta_{2k+1} \partial^{2k+1}_L \ ,
k\geq 0 \, , \\
\beta^{2k} = &  \ \partial^{e,2k-1}_H \theta_{2k-1} \ ,
k\geq 1 \, .
\end{align*}
Note that $\al^1=\bt^1$.
$\alpha^*$ and $\beta^*$ are homomorphisms of complexes (see resp. \cite{CP} and \cite{CDNo}). Let $\cone(\al^*)$
and $\cone(-\bt^*)$ be the mapping cones and $C^*(P,M)=\cone(\al^*)\bigsqcup_{(i_1,i_2)}\cone(-\bt^*)$ the pushout, where $i_1$ and $i_2$ are the following injections of complexes
\[\xymatrix{\cone\,(\al^*)&C_H^{*-1}(P,M^e)\ar[l]_-{i_1}
\ar[r]^-{i_2}&\cone\,(-\bt^*)}.\]
Define
$C^0_{\NP^l}(P,M)=0$, $C^1_{\NP^l}(P,M)=\Hom(P,M)$,
$C_{\NP^l}^n(P,M)=C^n(P,M)$, $n\geq 2$;
$\pa_{\NP^l}^0=0$, $\pa_{\NP^l}^1=(\pa_H^1,0,\pa_L^1)$,
$\pa_{\NP^l}^n=\pa^n$, $n\geq 2$.
We have $\pa_{\NP^{l}}^{n+1}\,\pa_{\NP^l}^n=0$, $n\geq 0$, so
$\{C_{\NP^l}^n(P,M),\pa_{\NP^l}^n,n\geq 0\}$ is a complex
which has the form
\[  \xymatrix{&&0\ar[d]&&\\
&&\hskip-5mm \Hom(P,M)\ar[dll]_-{-\pa_H^1}\ar[d]^0\ar[drr]^-{-\pa_L^1}\hskip-5mm&&\\
C_H^2(P,M)\ar[d]_-{-\pa_H^2}\ar[drr]^{\al^2}\hskip-5mm&\hskip-5mm\oplus
\hskip-5mm&\hskip-5mm
C_H^1(P,M^e)\ar[d]^-{\pa_H^{e,1}}\hskip-5mm&\hskip-5mm\oplus\hskip-5mm&
\hskip-5mm C_L^2(P,M)\ar[dll]_-{-\bt^2}\ar[d]^-{-\pa_L^2}\\
&&&&}\]
\vskip-5mm
\[\xymatrix{
C_H^3(P,M)\ar[d]_-{-\pa_H^3}\ar[drr]^{\al^3}\hskip-5mm&\hskip-5mm\oplus
\hskip-5mm&\hskip-5mm C_H^2(P,M^e)\ar[d]^{\pa_H^{e,2}}\hskip-5mm
&\hskip-5mm\oplus\hskip-5mm&\hskip-5mm C_L^3(P,M)\ar[dll]_-{-\bt^3}\ar[d]^-{-\pa_L^3}\\
\vdots&\vdots&\vdots&\vdots&\vdots}\] \vskip+5mm The cohomology
vector spaces $H_{\NP^l}^n(P,M)$, $n\geq 0$, of an
$\NP^l$-algebra $P$ with coefficients in a representation $M$ of
$P$ are defined by
\[  H_{{\NP^l}}^n(P,M)=H^n(C_{\NP^l}^*(P,M),\pa_{\NP^l}^n),\;\;\;n\geq 0. \]
According to \cite{CP} the cohomology of $\AWB$ is defined by
$H_{\AWB}^{n-1}(P,M)= H^n\big(\cone(\alpha^*)\big)$, for $n\geq 1$,
where $P\in \AWB$.  Note that in $\cone(\alpha^*)$ the  zero term
$\cone(\alpha^*)^0$ is zero, and the first one is $C_H^1(P,M)\oplus
C_H^0(P,M^e)$. In this paper the cohomology of $\AWB^l$ are defined as
$H_{\AWB^l}^0(P,M)=0$ and  $H_{\AWB^l}^n(P,M)=
H^n\big(\cone(\alpha^*)\big)$, for $n\geq 1$.

Now we shall define the cohomology vector spaces of $\NP^r$ and
$\NP^{lr}$-algebras. Let $\theta'_n \colon C^{n+1}_H(P,M) \to
C^n_H(P,M^e)$, for $n\geq 1$, be the homomorphism defined by
$\theta'_1(f)(p_1)(p_2)=f(p_2,p_1)$ and $\theta'_n(f)(p_1,
p_2,\dots, p_n)(p_{n+1})=f(p_{n+1},p_1,\dots,p_n), n>1$. It is easy
to see that $\theta'_n$ is an isomorphism for each $n \geq 1$.
Define the homomorphisms
\begin{align*}
&\al'^{*}  \colon  \bar{C}^*_H(P,M)\to
\bar{C}^*_H(P,M^e) \;,\\
&\bt'^{*} \colon \bar{C}^*_L(P,M)\to \bar{C}^*_H(P,M^e) \;,
\end{align*}
by
\[\alpha'^1(f)(p_1)(p_2) = [f(p_2),p_1]+[p_2,f(p_1)]-f([p_2,p_1])\]
and for $n>1$ by
\begin{multline*}\alpha'^n(f)(p_1,\dots,p_n)(p_{n+1})=[p_{n+1},f(p_1,\dots,p_n)] -
f([p_{n+1},p_1],p_2,\dots,p_n) \\ - f(p_1,[p_{n+1},p_2],\dots,p_n) - \dots -
f(p_1,\dots,p_{n-1},[p_{n+1},p_n]),
\end{multline*}
\begin{align*}
\beta'^{2k+1}   =  \ \theta'_{2k} \partial^{2k+1}_L, \
k\geq 0  \; , \qquad \qquad
\beta'^{2k} =  \ \partial^{e,2k-1}_H \theta'_{2k-1}, \
k\geq 1 \;.
\end{align*}
We have  $\al'^1=\bt'^1$.
Easy checking shows that $\alpha'^{*}$ and $\beta'^{*}$ are homomorphisms of complexes.

By taking the pushout
$C'^{*}(P,M)=\cone(\al'^{*})\bigsqcup_{(i'_1,i'_2)}\cone(-\bt'^{*})$,
where $i'_1$ and $i'_2$ are the following injections of complexes
\[\xymatrix{\cone\,(\al'^{*})&C_H^{*-1}(P,M^e)\ar[l]_-{i'_1}
\ar[r]^-{i'_2}&\cone\,(-\bt'^{*})\; , }\]
we construct the complex analogous to $\{C_{\NP^l}^n(P,M),\pa_{\NP^l}^n,n\geq 0\}$,
which will be denoted by $\{C_{\NP^r}^n(P,M),\pa_{\NP^r}^n,n\geq 0\}$. The complex has the form
\[  \xymatrix{&&0\ar[d]&&\\
&&\hskip-5mm \Hom(P,M)\ar[dll]_-{-\pa_H^1}\ar[d]^0\ar[drr]^-{-\pa_L^1}\hskip-5mm&&\\
C_H^2(P,M)\ar[d]_-{-\pa_H^2}\ar[drr]^{\al'^2}\hskip-5mm&\hskip-5mm\oplus
\hskip-5mm&\hskip-5mm
C_H^1(P,M^e)\ar[d]^-{\pa_H^{e,1}}\hskip-5mm&\hskip-5mm\oplus\hskip-5mm&
\hskip-5mm C_L^2(P,M)\ar[dll]_-{-\bt'^2}\ar[d]^-{-\pa_L^2}\\
&&&&}\]
\vskip-5mm
\[\xymatrix{
C_H^3(P,M)\ar[d]_-{-\pa_H^3}\ar[drr]^{\al'^3}\hskip-5mm&\hskip-5mm\oplus
\hskip-5mm&\hskip-5mm C_H^2(P,M^e)\ar[d]^{\pa_H^{e,2}}\hskip-5mm
&\hskip-5mm\oplus\hskip-5mm&\hskip-5mm C_L^3(P,M)\ar[dll]_-{-\bt'^3}\ar[d]^-{-\pa_L^3}\\
&&&&}\] \vskip-5mm
\[\xymatrix{
C_H^4(P,M)\ar[d]_-{-\pa_H^4}\ar[drr]^{\al'^4}\hskip-5mm&\hskip-5mm\oplus
\hskip-5mm&\hskip-5mm C_H^3(P,M^e)\ar[d]^{\pa_H^{e,3}}\hskip-5mm
&\hskip-5mm\oplus\hskip-5mm&\hskip-5mm C_L^4(P,M)\ar[dll]_-{-\bt'^4}\ar[d]^-{-\pa_L^4}\\
\vdots&\vdots&\vdots&\vdots&\vdots}\] \vskip+5mm

 The cohomology vector spaces of an $\NP^r$-algebra $P$ with coefficients in a representation $M$ of $P$ are defined as the cohomologies of this complex
 and denoted as $H_{\NP^{r}}^n(P,M), n\geq 0$.

Now we construct the complex for the cohomology of an $\NP^{lr}$-algebra $P$. Consider the following pairs of homomorphisms of complexes

\begin{align*}
&(\al^{*}, \al'^{*}) \colon \bar{C}^*_H(P,M)\to
\bar{C}^*_H(P,M^e)\oplus \bar{C}^*_H(P,M^e)\;,\\
&(\bt^{*}, \bt'^{*}) \colon \bar{C}^*_L(P,M)\to \bar{C}^*_H(P,M^e)\oplus \bar{C}^*_H(P,M^e) \;.
\end{align*}

 From these homomorphisms we obtain two cones: $\cone(\alpha^*,\alpha'^{*})$ and $\cone(\bt^*,\bt'^{*})$.  We have the following homomorphisms of complexes
\[\xymatrix{\cone(\al^*,\al'^{*})&C_H^{*-1}(P,M^e)\oplus C_H^{*-1}(P,M^e)\ar[l]_-{j_1}
\ar[r]^-{j_2}&\cone(-\bt^*,-\bt'^{*})}.\] The pushout of the pair
$(j_1,j_2)$ gives the desired complex. In particular, we take
$C^0_{\NP^{lr}}(P,M)=0$, $C^1_{\NP^{lr}}(P,M)=\Hom(P,M)$,
$C^n_{\NP^{lr}}(P,M)=C^n_H(P,M)\oplus C^n_H(P,M^e)\oplus
C^n_H(P,M^e)\oplus C^n_L(P,M)$, for $n\geq 2$;
$\pa_{\NP^{lr}}^0=0$, $\pa_{\NP^{lr}}^1=(-\pa_H^1, 0, 0, -\pa_L^1)$,
$\pa_{\NP^{lr}}^n$ is induced by $\alpha^n, \alpha'^{n},
\partial^{e,n-1}_H, \partial^{e,n-1}_H, \bt^n, \bt'^{n}$, for $n\geq
2$.

We have $\pa_{\NP^{lr}}^{n+1}\,\pa_{\NP^{lr}}^n=0$, for $n\geq
0$, therefore $\{C_{\NP^{lr}}^n(P,M),\pa_{\NP^{lr}}^n,n\geq 0\}$ is
a complex; it has the following form

\newpage

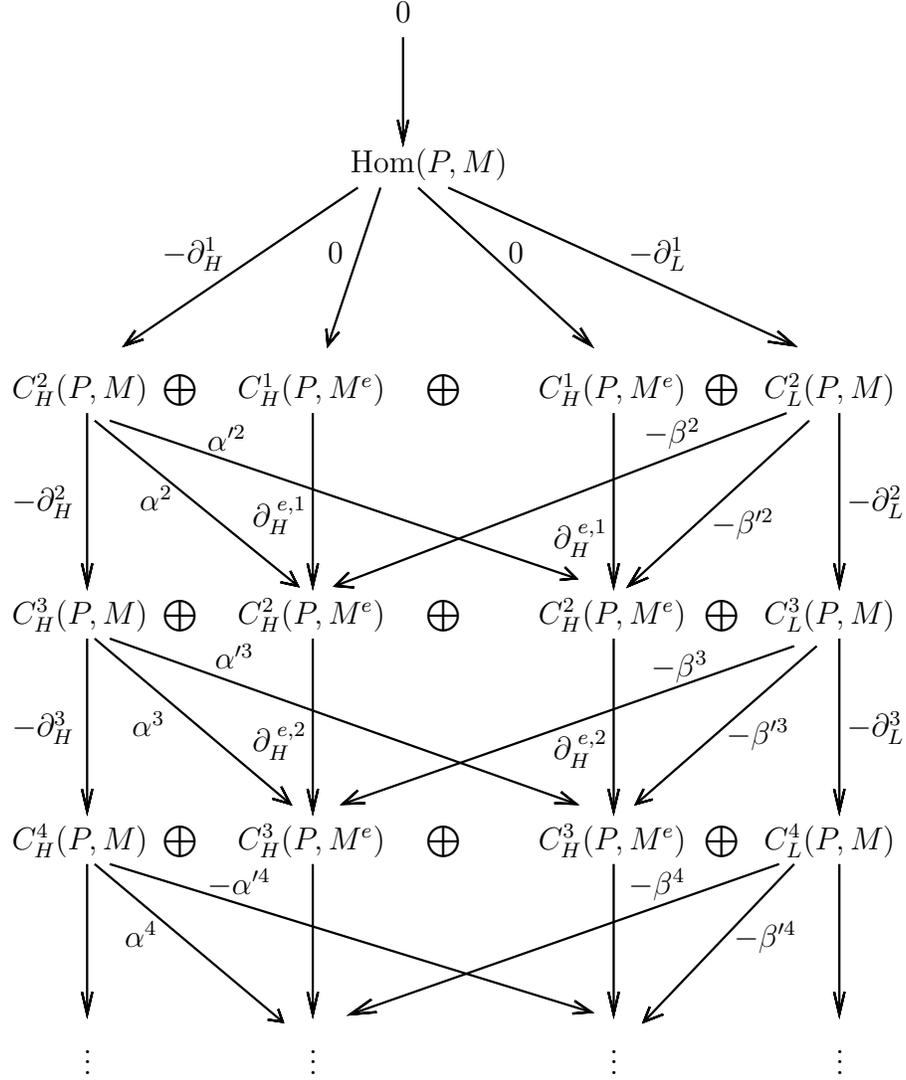
\begin{figure}[!h]
\centering
\begin{tikzpicture}[x=1.00mm, y=1.00mm, inner xsep=0pt, inner ysep=0pt, outer xsep=0pt, outer ysep=0pt]
\draw(101,110) node[anchor=base west]{\fontsize{12}{12}\selectfont 0};
\draw(95,90) node[anchor=base west]{\fontsize{12}{12}\selectfont $\Hom(P,M)$};
\draw(50,60) node[anchor=base west]{\fontsize{12}{12}\selectfont $C^2_H(P,M)$};
\draw(70,60) node[anchor=base west]{\fontsize{12}{12}\selectfont $\bigoplus$};
\draw(80,60) node[anchor=base west]{\fontsize{12}{12}\selectfont $C^1_H(P,M^e)$};
\draw(105,60) node[anchor=base west]{\fontsize{12}{12}\selectfont $\bigoplus$};
\draw(120,60) node[anchor=base west]{\fontsize{12}{12}\selectfont $C^1_H(P,M^e)$};
\draw(142,60) node[anchor=base west]{\fontsize{12}{12}\selectfont $\bigoplus$};
\draw(150,60) node[anchor=base west]{\fontsize{12}{12}\selectfont $C^2_L(P,M)$};
\draw(50,30) node[anchor=base west]{\fontsize{12}{12}\selectfont $C^3_H(P,M)$};
\draw(70,30) node[anchor=base west]{\fontsize{12}{12}\selectfont $\bigoplus$};
\draw(80,30) node[anchor=base west]{\fontsize{12}{12}\selectfont $C^2_H(P,M^e)$};
\draw(105,30) node[anchor=base west]{\fontsize{12}{12}\selectfont $\bigoplus$};
\draw(120,30) node[anchor=base west]{\fontsize{12}{12}\selectfont $C^2_H(P,M^e)$};
\draw(142,30) node[anchor=base west]{\fontsize{12}{12}\selectfont $\bigoplus$};
\draw(150,30) node[anchor=base west]{\fontsize{12}{12}\selectfont $C^3_L(P,M)$};
\draw(50,0) node[anchor=base west]{\fontsize{12}{12}\selectfont $C^4_H(P,M)$};
\draw(70,0) node[anchor=base west]{\fontsize{12}{12}\selectfont $\bigoplus$};
\draw(80,0) node[anchor=base west]{\fontsize{12}{12}\selectfont $C^3_H(P,M^e)$};
\draw(105,0) node[anchor=base west]{\fontsize{12}{12}\selectfont $\bigoplus$};
\draw(120,0) node[anchor=base west]{\fontsize{12}{12}\selectfont $C^3_H(P,M^e)$};
\draw(142,0) node[anchor=base west]{\fontsize{12}{12}\selectfont $\bigoplus$};
\draw(150,0) node[anchor=base west]{\fontsize{12}{12}\selectfont $C^4_L(P,M)$};
\definecolor{L}{rgb}{0,0,0}
\definecolor{F}{rgb}{0,0,0}
\path[line width=0.30mm, draw=L] (102,108) -- (102,94);
\path[line width=0.30mm, draw=L, fill=F] (102,94) -- (102.7,97) -- (102,94) -- (101.3,97) -- (102,94) -- cycle;
\path[line width=0.30mm, draw=L] (96,88) -- (65,67);
\path[line width=0.30mm, draw=L, fill=F] (65,67) -- (66.4,69) -- (65,67) -- (67.5,67.8) -- (65,67) -- cycle;
\draw(70,78) node[anchor=base west]{\fontsize{12}{12}\selectfont $-\partial^1_H$};
\path[line width=0.30mm, draw=L] (99,88) -- (92,67);
\path[line width=0.30mm, draw=L, fill=F] (92,67) -- (91.8,69.2) -- (92,67) -- (93.5,68.8) -- (92,67) -- cycle;
\draw(92,78) node[anchor=base west]{\fontsize{12}{12}\selectfont 0};
\path[line width=0.30mm, draw=L] (104,88) -- (127,67);
\path[line width=0.30mm, draw=L, fill=F] (127,67) -- (125.6,69.3) -- (127,67) -- (124.6,68.3) -- (127,67) -- cycle;
\draw(116,78) node[anchor=base west]{\fontsize{12}{12}\selectfont 0};
\path[line width=0.30mm, draw=L] (108,88) -- (154,67);
\path[line width=0.30mm, draw=L, fill=F] (154,67) -- (150.5,67.5) -- (154,67) -- (151.8,68.9) -- (154,67) -- cycle;
\draw(132,78) node[anchor=base west]{\fontsize{12}{12}\selectfont $-\partial^1_L$};
\path[line width=0.30mm, draw=L] (60,58) -- (60,35);
\path[line width=0.30mm, draw=L, fill=F] (60,35) -- (60.7,38) -- (60,35) -- (59.3,38) -- (60,35) -- cycle;
\draw(50,45) node[anchor=base west]{\fontsize{12}{12}\selectfont $-\partial^2_H$};
\path[line width=0.30mm, draw=L] (90,58) -- (90,35);
\path[line width=0.30mm, draw=L, fill=F] (90,35) -- (90.7,38) -- (90,35) -- (89.3,38) -- (90,35) -- cycle;
\draw(82,43) node[anchor=base west]{\fontsize{12}{12}\selectfont $\partial_H^{\; e,1}$};
\path[line width=0.30mm, draw=L] (130,58) -- (130,35);
\path[line width=0.30mm, draw=L, fill=F] (130,35) -- (130.7,38) -- (130,35) -- (129.3,38) -- (130,35) -- cycle;
\draw(122,40) node[anchor=base west]{\fontsize{12}{12}\selectfont $\partial_H^{\; e,1}$};
\path[line width=0.30mm, draw=L] (160,58) -- (160,35);
\path[line width=0.30mm, draw=L, fill=F] (160,35) -- (160.7,38) -- (160,35) -- (159.3,38) -- (160,35) -- cycle;
\draw(161,45) node[anchor=base west]{\fontsize{12}{12}\selectfont $-\partial^2_L$};
\path[line width=0.30mm, draw=L] (60,28) -- (60,5);
\path[line width=0.30mm, draw=L, fill=F] (60,5) -- (60.7,8) -- (60,5) -- (59.3,8) -- (60,5) -- cycle;
\draw(50,15) node[anchor=base west]{\fontsize{12}{12}\selectfont $-\partial^3_H$};
\path[line width=0.30mm, draw=L] (90,28) -- (90,5);
\path[line width=0.30mm, draw=L, fill=F] (90,5) -- (90.7,8) -- (90,5) -- (89.3,8) -- (90,5) -- cycle;
\draw(82,13) node[anchor=base west]{\fontsize{12}{12}\selectfont $\partial_H^{\;e,2}$};
\path[line width=0.30mm, draw=L] (130,28) -- (130,5);
\path[line width=0.30mm, draw=L, fill=F] (130,5) -- (130.7,8) -- (130,5) -- (129.3,8) -- (130,5) -- cycle;
\draw(122,12) node[anchor=base west]{\fontsize{12}{12}\selectfont $\partial_H^{\;e,2}$};
\path[line width=0.30mm, draw=L] (160,28) -- (160,5);
\path[line width=0.30mm, draw=L, fill=F] (160,5) -- (160.7,8) -- (160,5) -- (159.3,8) -- (160,5) -- cycle;
\draw(161,15) node[anchor=base west]{\fontsize{12}{12}\selectfont $-\partial^3_L$};
\path[line width=0.30mm, draw=L] (60,-2) -- (60,-22);
\path[line width=0.30mm, draw=L, fill=F] (60,-22) -- (60.7,-19) -- (60,-22) -- (59.3,-19) -- (60,-22) -- cycle;
\path[line width=0.30mm, draw=L] (90,-2) -- (90,-22);
\path[line width=0.30mm, draw=L, fill=F] (90,-22) -- (90.7,-19) -- (90,-22) -- (89.3,-19) -- (90,-22) -- cycle;
\path[line width=0.30mm, draw=L] (130,-2) -- (130,-22);
\path[line width=0.30mm, draw=L, fill=F] (130,-22) -- (130.7,-19) -- (130,-22) -- (129.3,-19) -- (130,-22) -- cycle;
\path[line width=0.30mm, draw=L] (160,-2) -- (160,-22);
\path[line width=0.30mm, draw=L, fill=F] (160,-22) -- (160.7,-19) -- (160,-22) -- (159.3,-19) -- (160,-22) -- cycle;
\draw(59.5,-30) node[anchor=base west]{\fontsize{12}{12}\selectfont $\vdots$};
\draw(89.5,-30) node[anchor=base west]{\fontsize{12}{12}\selectfont $\vdots$};
\draw(129.5,-30) node[anchor=base west]{\fontsize{12}{12}\selectfont $\vdots$};
\draw(159.5,-30) node[anchor=base west]{\fontsize{12}{12}\selectfont $\vdots$};
\path[line width=0.30mm, draw=L] (61,57) -- (88,35);
\path[line width=0.30mm, draw=L, fill=F] (88,35) -- (85.7,35.5) -- (88,35) -- (87,37) -- (88,35) -- cycle;
\draw(67,45) node[anchor=base west]{\fontsize{12}{12}\selectfont $\alpha^2$};
\path[line width=0.30mm, draw=L] (63,57) -- (125,36);
\path[line width=0.30mm, draw=L, fill=F] (125,36) -- (122.8,35.8) -- (125,36) -- (123.5,37.5) -- (125,36) -- cycle;
\draw(76,53) node[anchor=base west]{\fontsize{12}{12}\selectfont $\alpha'^{2}$};
\path[line width=0.30mm, draw=L] (153,58) -- (93,35);
\path[line width=0.30mm, draw=L, fill=F] (93,35) -- (94.8,36.8) -- (93,35) -- (95.5,34.8) -- (93,35) -- cycle;
\draw(134,54) node[anchor=base west]{\fontsize{12}{12}\selectfont $-\beta^2$};
\path[line width=0.30mm, draw=L] (156,57) -- (132,35);
\path[line width=0.30mm, draw=L, fill=F] (132,35) -- (133.5,37.5) -- (132,35) -- (134.6,36.1) -- (132,35) -- cycle;
\draw(143,42) node[anchor=base west]{\fontsize{12}{12}\selectfont $-\beta'^{2}$};
\path[line width=0.30mm, draw=L] (61,28) -- (87,6);
\path[line width=0.30mm, draw=L, fill=F] (87,6) -- (84.6,6.7) -- (87,6) -- (85.8,8.3) -- (87,6) -- cycle;
\draw(66,15) node[anchor=base west]{\fontsize{12}{12}\selectfont $\alpha^3$};
\path[line width=0.30mm, draw=L] (63,28) -- (125,6);
\path[line width=0.30mm, draw=L, fill=F] (125,6) -- (122.1,6.1) -- (125,6) -- (123,7.6) -- (125,6) -- cycle;
\draw(77,24) node[anchor=base west]{\fontsize{12}{12}\selectfont $\alpha'^{3}$};
\path[line width=0.30mm, draw=L] (154,27) -- (94,6);
\path[line width=0.30mm, draw=L, fill=F] (94,6) -- (96,7.8) -- (94,6) -- (97,6.1) -- (94,6) -- cycle;
\draw(135,23) node[anchor=base west]{\fontsize{12}{12}\selectfont $-\beta^3$};
\path[line width=0.30mm, draw=L] (157,27) -- (133,6);
\path[line width=0.30mm, draw=L, fill=F] (133,6) -- (134,8) -- (133,6) -- (135.2,6.6) -- (133,6) -- cycle;
\draw(145,14) node[anchor=base west]{\fontsize{12}{12}\selectfont $-\beta'^{3}$};
\path[line width=0.30mm, draw=L] (154,-2) -- (134,-23);
\path[line width=0.30mm, draw=L, fill=F] (134,-23) -- (134.8,-21) -- (134,-23) -- (136.6,-22) -- (134,-23) -- cycle;
\draw(146,-13) node[anchor=base west]{\fontsize{12}{12}\selectfont $-\beta'^{4}$};
\path[line width=0.30mm, draw=L] (152,-2) -- (95,-22);
\path[line width=0.30mm, draw=L, fill=F] (95,-22) -- (97.8,-20) -- (95,-22) -- (98.5,-21.7) -- (95,-22) -- cycle;
\draw(132,-6) node[anchor=base west]{\fontsize{12}{12}\selectfont $-\beta^4$};
\path[line width=0.30mm, draw=L] (63,-2) -- (127,-21.7);
\path[line width=0.30mm, draw=L, fill=F] (127,-21.7) -- (124.5,-21.9) -- (127,-21.7) -- (124.8,-20.2) -- (127,-21.7) -- cycle;
\draw(76,-6) node[anchor=base west]{\fontsize{12}{12}\selectfont $-\alpha'^{4}$};
\path[line width=0.30mm, draw=L] (61,-2) -- (86,-23);
\path[line width=0.30mm, draw=L, fill=F] (86,-23) -- (84,-22.7) -- (86,-23) -- (85.4,-21.3) -- (86,-23) -- cycle;
\draw(65,-13) node[anchor=base west]{\fontsize{12}{12}\selectfont $\alpha^4$};
\end{tikzpicture}
\end{figure}

The cohomology vector spaces of an $\NP^{lr}$-algebra $P$ with
coefficients in a representation $M$ of $P$ are defined as the
cohomologies of this complex
 and denoted as $H_{\NP^{lr}}^n(P,M), n\geq 0$.

As in the case of $\NP^l$-algebras in \cite{CDNo}, we define
restricted second cohomology  of $\NP^{lr}$-algebras. We have the
natural
 injection $C_H^2(P,M)\oplus C_L^2(P,M)\lra C_{\NP^{lr}}^2(P,M)$ on to the first and the fourth summands; the image of this injection will be denoted again by the
sum $ C_H^2(P,M)\oplus C_L^2(P,M)$. Consider the
restriction
\[ d_{\NP^{lr}}^2=\pa_{\NP^{lr}}^2 \big\vert_{{C_H^2(\;)\oplus C_L^2(\;)}}. \]
We define the 2-dimensional restricted cohomology of the
$\NP^{lr}$-algebra $P$ with coefficients in $M$ by
\[  \bH_{\NP^{lr}}^2(P,M)=\Ker d_{\NP^{lr}}^2/\Imm\pa_{\NP^{lr}}^1 \,. \]
The obvious injection \[\kappa \colon \Ker d_{\NP^{lr}}^2\lra \Ker
\pa_{\NP^{lr}}^2\]
induces the injection of the corresponding
cohomologies
\[{\chi} \colon \bH_{\NP^{lr}}^2(P,M) \lra H_{\NP^{lr}}^2(P,M).\]

$\bH_{\NP^{r}}^2(P,M)$ is defined in analogous way as for $\NP^l$-algebras.

The cohomologies of $\AWB^r$ and $\AWB^{lr}$ are defined by
\begin{align*}H^*_{\AWB^r}(P,M)= &  \ H^*\big(\cone(\alpha'^{*})\big), \\
H^*_{\AWB^{lr}}(P,M)= &  \ H^*\big(\cone(\alpha^*,\alpha'^{*})\big).
\end{align*}
From the definitions we obtain
\begin{lemma}\
\begin{itemize}
  \item[(i)] For $P\in \mathbf{NP}$ we have
\begin{align*}H_{{\NP}}^0(P,M)= &  \  0, \\
H_{{\NP}}^1(P,M) = &  \ \Der_{{\NP}}(P,M)\, .
\end{align*}
  \item[(ii)] For $P\in \mathbf{AWB}$ we have
\begin{align*}
H_{{\AWB}}^0(P,M)= &  \  0, \\
H_{{\AWB}}^1(P,M)= & \  \Der_{{\AWB}}(P,M)\,,  \quad  \mbox{and} \\
H_{{\AWB}}^2(P,M)\cong  &  \ \Ext_{{\AWB}}(P,M) \, .
\end{align*}
\end{itemize}
\end{lemma}
\begin{proof}(i) The proof follows directly from the fact that
$C^0_{\NP}(P,M)=0$,  from the definition of
$\partial^1_{\NP}$ and the definition of a derivation.

   (ii) Since the zero term in the corresponding cone complex is zero, the first equality follows from the definition of the cohomology.
  The proofs of other two equalities of (ii) for $\AWB^r$ and $\AWB^{lr}$ are similar to the proofs given in \cite{CP} for $\AWB^l$.

\end{proof}
\begin{theorem}\label{h2_ext}
$\bH_{{\NP}}^2(P,M) \cong \Ext_{{\NP}}(P,M)$.
\end{theorem}
\begin{proof} The proof is similar to the one for $\NP^l$-algebras presented in \cite{CDNo}.
\end{proof}
\begin{corollary} \label{C:free_twozero} \
\begin{itemize}
  \item[(i)] If $P$ is a free $\NP$-algebra, then \[\bH_{{\NP}}^2(P,M)=0 \; \]
 for any representation $M$ of $P$.
  \item[(ii)] If $P$ is a free $\NP^l$  (resp. $\NP^r$-algebra), then
   \[H_{{\NP}^l}^n(P,M)=0  \quad (resp. \; H_{{\NP}^r}^n(P,M)=0) \; ,\]
    for $n\geq 3$ and any representation $M$.
\end{itemize}
\end{corollary}

\begin{proof}(i) Since for a free $\NP$-algebra $P$ every extension $0 \to M \xrightarrow{i}
Q \xrightarrow{j} P \to 0$ splits, the fact follows from Theorem \ref{h2_ext}.

 (ii) From Proposition \ref{free_aL} (resp. Proposition \ref{P:free_npr}) it follows that $V^l_{A}(P)$ \big(resp. $V^r_{A}(P)$\big) is a free associative
algebra and $V^l_{L}(P)$ \big(resp. $V^r_{L}(P)$\big) is a free Leibniz algebra. It is well known
that cohomologies of free associative algebras and free Leibniz
algebras vanish in dimensions $\geq 2$ \cite{LP}. Thus we have
$H_H^n(P,-)=0$ and $H_L^n(P,-)=0$ for $n\geq
2$. From this and from the fact that $\alpha^*$ and $\beta^*$ (resp. $\alpha'^*$ and $\beta'^*$) are
homomorphisms of cochain complexes, by diagram chasing  we obtain that
$C^*_{\NP^l}(P,M)$ \big(resp. $C^*_{\NP^r}(P,M)$\big) is exact in dimensions $>2$ for a
free $\NP^l$-algebra (resp. $\NP^r$-algebra) $P$.
\end{proof}
\begin{lemma}\label{free_AWB}
If $P$ is a free $\AWB^r$, then \[H_{\AWB^r}^n(P,M)=0 \;,\] for
$n\geq 2$ (according to the notation in \cite{CP}, $n\geq 1$) and
any representation $M$ of $P$.
\end{lemma}
The proof is analogous to the proof of this fact for $\AWB^l$ given in \cite{CP} and therefore it is omitted.

In \cite{CP} it is proved that if $P$ is $\AWB^l$, then its cohomologies are isomorphic to Quillen cohomologies. In the similar way, applying Lemma \ref{free_AWB} we have
\begin{theorem}
$H_{\AWB^r}^{n+1}(P,M) \approx H_{\mathrm Q}^n(P,M)$.
\end{theorem}

From the constructions of the cohomology complexes we obtain the following short exact sequences of complexes

\begin{align*}
&0 \longrightarrow \cone(\alpha^*) \longrightarrow  C^*_{\NP^l}(P,M) \longrightarrow  C^*_L(P,M)  \longrightarrow 0 \, , \quad  * \geq 3  \tag{$a_1$}\\
&0 \longrightarrow \cone(\alpha'^*) \longrightarrow  C^*_{\NP^r}(P,M) \longrightarrow  C^*_L(P,M)  \longrightarrow 0 \, ,  \quad  * \geq 3  \tag{$a_2$}\\
&0 \longrightarrow \cone(\alpha^*,\alpha'^*) \longrightarrow  C^*_{\NP^{lr}}(P,M) \longrightarrow  C^*_L(P,M)  \longrightarrow 0 \, , \quad  * \geq 3  \tag{$a$} \\
&0 \longrightarrow \cone(-\beta^{*}) \longrightarrow  C^*_{\NP^l}(P,M) \longrightarrow  C^*_H(P,M)  \longrightarrow 0 \, , \quad  * \geq 3  \tag{$b_1$} \\
&0 \longrightarrow \cone(-\beta'^*) \longrightarrow  C^*_{\NP^r}(P,M) \longrightarrow  C^*_H(P,M)  \longrightarrow 0 \, , \quad  * \geq 3  \tag{$b_2$}\\
&0 \longrightarrow \cone(-\beta^*, -\beta'^*) \longrightarrow  C^*_{\NP^{lr}}(P,M) \longrightarrow  C^*_H(P,M)  \longrightarrow 0 \, , \quad  * \geq 1  \tag{$b$}   \\
&0 \longrightarrow C^{*-1}_H(P,M^e) \longrightarrow   C^{*}_{\AWB^l}(P,M) \longrightarrow  C^*_H(P,M)  \longrightarrow 0  \, , \quad  * \geq 1  \tag{$c_1$} \\
&0 \longrightarrow C^{*-1}_H(P,M^e) \longrightarrow   C^{*}_{\AWB^r}(P,M) \longrightarrow  C^*_H(P,M)  \longrightarrow 0 \, , \quad  * \geq 1 \tag{$c_2$} \\
&0 \longrightarrow C^{*-1}_H(P,M^e) \xrightarrow{ \ i_3 \ }   C^{*}_{\AWB^{lr}}(P,M) \longrightarrow  C^*_{\AWB^l}(P,M)  \longrightarrow 0  \, , \quad  * \geq 1  \tag{$c$}\\
&0 \longrightarrow C^{*-1}_H(P,M^e) \xrightarrow{ \ i_2 \ }   C^{*}_{\AWB^{lr}}(P,M) \longrightarrow  C^*_{\AWB^r}(P,M)  \longrightarrow 0  \, , \quad  * \geq 1 \tag{$c'$}
\end{align*}
\begin{align*}
&0 \to C^{*-1}_H(P,M^e) \to  C^{*}_{\NP^l}(P,M) \to  C^*_H(P,M) \oplus  C^*_L(P,M) \to 0 \, , \  * \geq 3  \tag{$d_1$} \\
&0 \to C^{*-1}_H(P,M^e) \to  C^{*}_{\NP^r}(P,M) \to  C^*_H(P,M) \oplus  C^*_L(P,M) \to 0 \, ,  \ * \geq 3  \tag{$d_2$} \\
&0 \longrightarrow C^{*-1}_H(P,M^e) \xrightarrow{ \ i_3 \ } C^*_{\NP^{lr}}(P,M) \longrightarrow  C^*_{\NP^l}(P,M)  \longrightarrow 0 \, , \quad  * \geq 3  \tag{$d$} \\
&0 \longrightarrow C^{*-1}_H(P,M^e) \xrightarrow{ \ i_2 \ }   C^{*}_{\NP^{lr}}(P,M) \longrightarrow   C^*_{\NP^r}(P,M) \longrightarrow 0 \, , \quad  * \geq 3  \tag{$d'$} \label{d}\\
&0 \to C^{*-1}_H(P,M^e) \oplus C^{*-1}_H(P,M^e) \to  C^*_{\AWB^{lr}}(P,M) \to  C^*_H(P,M)  \to 0 \, , \  * \geq 3  \tag{$e$}
\end{align*}
\begin{align*}
& 0 \to C^{*-1}_H(P,M^e) \oplus C^{*-1}_H(P,M^e)  \to
  C^*_{\NP^{lr}}(P,M) \to C^*_H(P,M)  \oplus C^*_L(P,M) \to 0 \tag{$f$}
\end{align*}
\begin{align*}
&0 \to C^{*-1}_H(P,M^e) \xrightarrow{(i_2 ,-i_3)}   \cone(\alpha^{*})\oplus  \cone(-\beta^{*}) \to
  C^*_{\NP^l}(P,M)  \to  0 \, , \ * \geq 3  \tag{$g_1$}
\end{align*}
\begin{align*}
&0 \to C^{*-1}_H(P,M^e) \xrightarrow{(i_2 ,-i_3)}  \cone(\alpha'^*)\oplus  \cone(-\beta'^*) \to
 C^*_{\NP^r}(P,M)  \to 0 \,  , \ * \geq 3  \tag{$g_2$}
\end{align*}
\begin{align*}
& 0 \to C^{*-1}_H(P,M^e) \oplus C^{*-1}_H(P,M^e) \xrightarrow{((i_2 ,i_3),(-i_4 ,-i_5))}   \cone(\alpha^{*}, \alpha'^*)\oplus \cone(-\beta^{*}, -\beta'^*) \\ & \to C^*_{\NP^{lr}}(P,M) \to 0,  \tag{$g$} \label{g}
\end{align*}
 \begin{align*}
&0 \to C^{*-1}_H(P,M^e) \to  \cone(-\beta^{*}) \to  C^*_L(P,M)  \to 0 \, , \  * \geq 3  \tag{$h_1$} \\
&0 \longrightarrow C^{*-1}_H(P,M^e) \longrightarrow  \cone(-\beta'^*) \longrightarrow  C^*_L(P,M)  \longrightarrow 0 \,  , \quad  * \geq 3  \tag{$h_2$}
\end{align*}
\begin{align*}
& 0 \to  C^{*-1}_H(P,M^e)\oplus C^{*-1}_H(P,M^e) \to
 \cone(-\beta^{*}, -\beta'^*) \to
 C^*_L(P,M) \to  0, \ * \geq 3.     \tag{$h$}
\end{align*}
In these sequences $i_2,i_3,i_4$ and $i_5$ denote the
injections on the corresponding summands, respectively. These exact
sequences are obtained directly from  the constructions of the
cohomology complexes of the corresponding types of algebras.

\begin{theorem} \label{T:ex_seq}
We have the following exact sequences of cohomology vector spaces

\begin{equation}\label{A_1}
\xymatrix{& H^2_{\AWB^l}(P,M) \ar[r] & H^2_{\NP^l}(P,M)  \ar[r] & H^2_{L}(P,M)  \ar`r[d] `[l]
`[llld] `[d] [dll]\\
& H^3_{\AWB^l}(P,M) \ar[r] & H^3_{\NP^l}(P,M)  \ar[r] & H^3_{L}(P,M)  \ar[r] & \cdots } \tag{$A_1$}
\end{equation}
where $P$ is an $\NP^l$-algebra and $M$ a representation of $P$.

 \begin{equation}\label{A_2}
\xymatrix{ & H^2_{\AWB^r}(P,M) \ar[r] & H^2_{\NP^r}(P,M)  \ar[r] & H^2_{L}(P,M)  \ar`r[d] `[l]
`[llld] `[d] [dll]\\
& H^3_{\AWB^r}(P,M) \ar[r] & H^3_{\NP^r}(P,M)  \ar[r] & H^3_{L}(P,M)  \ar[r] & \cdots } \tag{$A_2$}
\end{equation}
where $P$ is an $\NP^r$-algebra and $M$ a representation of $P$.

 \begin{equation}\label{A}
\xymatrix{ & H^2_{\AWB^{lr}}(P,M) \ar[r] & H^2_{\NP^{lr}}(P,M)  \ar[r] & H^2_{L}(P,M)  \ar`r[d] `[l]
`[llld] `[d] [dll]\\
& H^3_{\AWB^{lr}}(P,M) \ar[r] & H^3_{\NP^{lr}}(P,M)  \ar[r] &
H^3_{L}(P,M)  \ar[r] & \cdots } \tag{$A$}
\end{equation}
where $P$ is an $\NP^{lr}$-algebra and $M$ a representation of $P$.

 \begin{equation}\label{B_1}
\xymatrix{ & H^3(\cone(-\beta^{*})) \ar[r] & H^3_{\NP^{l}}(P,M)  \ar[r] & H^3_{H}(P,M)  \ar`r[d] `[l]
`[llld] `[d] [dll]\\
& H^4(\cone(-\beta^{*})) \ar[r] & H^4_{\NP^{l}}(P,M)  \ar[r] & H^4_{H}(P,M)  \ar[r] & \cdots } \tag{$B_1$}
\end{equation}
where $P$ is an $\NP^l$-algebra and $M$ a representation of $P$.

  \begin{equation}\label{B_2}
\xymatrix{ & H^3(\cone(-\beta'^{*})) \ar[r] & H^3_{\NP^{r}}(P,M)  \ar[r] & H^3_{H}(P,M)  \ar`r[d] `[l]
`[llld] `[d] [dll]\\
& H^4(\cone(-\beta'^{*})) \ar[r] & H^4_{\NP^{r}}(P,M)  \ar[r] & H^4_{H}(P,M)  \ar[r] & \cdots } \tag{$B_2$}
\end{equation}
where $P$ is an $\NP^r$-algebra and $M$ a representation of $P$.

  \begin{equation}\label{B}
\xymatrix@C=5mm{ & H^3 (\cone(-\beta^*, -\beta'^*)) \ar[r] & H^3_{\NP^{lr}}(P,M)  \ar[r] & H^3_{H}(P,M)  \ar`r[d] `[l]
`[llld] `[d] [dll]\\
& H^4(\cone(-\beta^*, -\beta'^*)) \ar[r] & H^4_{\NP^{lr}}(P,M)  \ar[r] & H^4_{H}(P,M)  \ar[r] & \cdots } \tag{$B$}
\end{equation}
where $P$ is an $\NP^{lr}$-algebra and $M$ a representation of $P$.

  \begin{equation}\label{C_1,2}
\xymatrix{ & H^2_H(P,M^e) \ar[r] & H^3_{\AWB^{r}}(P,M)  \ar[r] & H^3_{H}(P,M)  \ar`r[d] `[l]
`[llld] `[d] [dll]\\
& H^3_H(P,M^e) \ar[r] & H^4_{\AWB^{r}}(P,M)   \ar[r] & H^4_{H}(P,M)  \ar[r] & \cdots } \tag{$C_{1,2}$}
\end{equation}
where $P$ is an $\AWB^r$ and $M$ a representation of $P$. Analogous
exact sequence we have for $H_{\AWB^{l}}(P,M)$.

 \begin{equation}\label{C}
\xymatrix{ & H^2_H(P,M^e) \ar[r] & H^3_{\AWB^{lr}}(P,M)  \ar[r] & H^3_{\AWB^{r}}(P,M)  \ar`r[d] `[l]
`[llld] `[d] [dll]\\
& H^3_H(P,M^e) \ar[r] & H^4_{\AWB^{lr}}(P,M)   \ar[r] & H^4_{\AWB^{r}}(P,M)  \ar[r] & \cdots } \tag{$C, C'$}
\end{equation}
where $P$ is an $\AWB^{lr}$ and $M$ a representation of $P$.
Analogous exact sequence we have, where $H_{\AWB^{r}}(P,M)$ is
replaced by $H_{\AWB^{l}}(P,M)$.

 \begin{equation}\label{d_1,2}
\xymatrix@C=5mm{ & H^2_H(P,M^e) \ar[r] & H^3_{\NP^{r}}(P,M)  \ar[r] & H^3_{H}(P,M) \oplus H^3_{L}(P,M) \ar`r[d] `[l]
`[llld] `[d] [dll]\\
& H^3_H(P,M^e) \ar[r] & H^4_{\NP^{r}}(P,M)   \ar[r] & H^4_{H}(P,M) \oplus H^4_{L}(P,M)  \ar[r] & \cdots } \tag{$D_{1,2}$}
\end{equation}
where $P$ is an $\NP^r$-algebra and $M$ a representation of $P$.
Analogously for $H_{\NP^{l}}(P,M)$.

\begin{equation}\label{dd'}
\xymatrix{& H^2_H(P,M^e) \ar[r] & H^3_{\NP^{lr}}(P,M)  \ar[r] & H^3_{\NP^{r}}(P,M)  \ar`r[d] `[l]
`[llld] `[d] [dll]\\
& H^3_H(P,M^e) \ar[r] & H^4_{\NP^{lr}}(P,M)   \ar[r] &  H^4_{\NP^{r}}(P,M)  \ar[r] & \cdots } \tag{$D,D'$}
\end{equation}
where $P$ is an $\NP^{lr}$-algebra and $M$ a representation of $P$.
Analogous exact sequence we have, where $H_{\NP^{r}}(P,M)$ is
replaced by $H_{\NP^{l}}(P,M)$.

\begin{equation}\label{e}
\xymatrix@C=5mm{ & H^2_H(P,M^e)\oplus H^2_{H}(P,M^e) \ar[r] & H^3_{\AWB^{lr}}(P,M)  \ar[r] & H^3_H(P,M)  \ar`r[d] `[l]
`[llld] `[d] [dll]\\
& H^3_H(P,M^e)\oplus H^3_{H}(P,M^e) \ar[r] & H^4_{\AWB^{lr}}(P,M)   \ar[r] &  H^4_H(P,M)  \ar[r] & \cdots } \tag{$E$}
\end{equation}
where $P$ is an $\AWB^{lr}$ and $M$ a representation of $P$.

\begin{equation}\label{f}
\xymatrix@C=3.5mm{  & H^2_H(P,M^e)\oplus H^2_{H}(P,M^e) \ar[r] & H^3_{\NP^{lr}}(P,M)  \ar[r] & H^3_H(P,M) \oplus H^3_L(P,M) \ar`r[d] `[l]
`[llld] `[d] [dll]\\
& H^3_H(P,M^e)\oplus H^3_{H}(P,M^e) \ar[r] & H^4_{\NP^{lr}}(P,M)   \ar[r] &  H^4_H(P,M) \oplus H^4_L(P,M)  \ar[r] & \cdots } \tag{$F$}
\end{equation}
where $P$ is an $\NP^{lr}$-algebra and $M$ a representation of $P$.

\begin{equation}\label{g_1}
\xymatrix@C=3.5mm{  & H^2_H(P,M^e) \ar[r] & H^3_{\AWB^{l}}(P,M) \oplus H^3 (\cone(-\beta^*)) \ar[r] & H^3_{\NP^{l}}(P,M) \ar`r[d] `[l]
`[llld] `[d] [dll]\\
& H^3_H(P,M^e) \ar[r] & H^4_{\AWB^{l}}(P,M) \oplus H^4 (\cone(-\beta^*))   \ar[r] &  H^4_{\NP^{l}}(P,M)  \ar[r] & \cdots } \tag{$G_1$}
\end{equation}
where $P$ is an $\NP^l$-algebra and $M$ a representation of $P$.

\begin{equation}\label{g_2}
\xymatrix@C=3.5mm{ & H^2_H(P,M^e) \ar[r] & H^3_{\AWB^{r}}(P,M) \oplus H^3 (\cone(-\beta'^*)) \ar[r] & H^3_{\NP^{r}}(P,M) \ar`r[d] `[l]
`[llld] `[d] [dll]\\
& H^3_H(P,M^e) \ar[r] & H^4_{\AWB^{r}}(P,M) \oplus H^4 (\cone(-\beta'^*))   \ar[r] &  H^4_{\NP^{r}}(P,M)  \ar[r] & \cdots } \tag{$G_2$}
\end{equation}
where $P$ is an $\NP^r$-algebra and $M$ a representation of $P$.
\begin{align*}
 & H^2_H(P,M^e)\oplus H^2_{H}(P,M^e) \to H^3_{\AWB^{lr}}(P,M) \oplus H^3 (\cone(-\beta^*,-\beta'^*)) \to \\
         &  \to H^3_{\NP^{lr}}(P,M) \to  H^3_H(P,M^e) \oplus H^3_{H}(P,M^e) \to \\
         & \to H^4_{\AWB^{lr}}(P,M) \oplus H^4 (\cone(-\beta^*,-\beta'^*)) \to H^4_{\NP^{lr}}(P,M) \to \cdots  \tag{$G$}
\end{align*}
where $P$ is an $\NP^{lr}$-algebra and $M$ a representation of $P$.

\begin{equation}\label{h_1,2}
\xymatrix{ & H^2_H(P,M^e) \ar[r] & H^3 (\cone(-\beta^*))  \ar[r] &  H^3_L(P,M) \ar`r[d] `[l]
`[llld] `[d] [dll]\\
& H^3_H(P,M^e) \ar[r] & H^4 (\cone(-\beta^*))   \ar[r] &   H^4_L(P,M)  \ar[r] & \cdots } \tag{$H_{1,2}$}
\end{equation}
where $P$ is an $\NP^l$-algebra and $M$ a representation of $P$.
Analogous exact sequence we have for the cohomologies of the
$\cone(-\beta'^*)$ and for an $\NP^r$-algebra $P$.
\begin{align*}\label{h}
& H^2_H(P,M^e)\oplus H^2_H(P,M^e) \to H^3 (\cone(-\beta^*,-\beta'^*)) \to H^3_L(P,M)   \tag{$H$} \to \\
&  \to H^3_H(P,M^e)\oplus H^3_H(P,M^e) \to H^4 (\cone(-\beta^*,-\beta'^*))   \to  H^4_L(P,M)  \to \cdots
\end{align*}
for any $\NP^{lr}$-algebra $P$ and a representation $M$ of $P$.
\end{theorem}
\begin{proof}
These exact sequences are obtained directly from the corresponding
short exact sequences of the cohomology complexes.
\end{proof}

Recall that  \emph{Hochschild cohomological dimension}
$\cdim_H P$ of an associative algebra $P$ is defined
as the greatest natural number  $n$, for which there exists a $P$-$P$-bimodule $S$ with $H^n_H(P,S)\neq 0$.
The analogous meaning will have
{\it Leibniz cohomological dimension} of a Leibniz algebra $P$,
{\it $\AWB$ cohomological dimension} of an algebra $P\in \AWB$ and
{\it $\NP$ cohomological dimension} of an $\NP$-algebra $P$, denoted
as $\cdim_L P$, $\cdim_{\AWB} P$ and
$\cdim_{\NP} P$, respectively.

\begin{corollary}\label{C:free_zero}
Let $P$ be a free $\NP^r$-algebra (resp.  $\NP^l$-algebra) and $M$
 be a representation of $P$. Then we have
\begin{itemize}
  \item[(i)]  $H^n_{\AWB^{r}}(P,M)  = 0 \, , n \geq 3$
             (resp. $H^n_{\AWB^{l}}(P,M) = 0 \, , n \geq 3$), where $P$ is the underlying $\AWB^{r}$ (resp. $\AWB^{l}$) of the given algebra $P$;
 \item[(ii)] $H^n\big(\cone(-\beta'^*)\big) = 0$ \,
             \big(resp. $H^n\big(\cone(-\beta^{*})\big) = 0$\big), $n \geq 3$.
 \end{itemize}
\end{corollary}

\begin{proof} (i) Since $P$ is a free $\NP^r$-algebra, by Corollary  \ref{C:free_twozero} (ii) $H_{\NP^r}^n(P,M)=0, \, n\geq 3$.  By Proposition \ref{P:free_npr} the underlying Leibniz algebra is also free, from  which follows the well-known fact that $H^n_L(P,N)=0$, $n\geq 2$, for any representation $N$ of $P$ in the category of Leibniz algebras $\mathbf{Leib}$, i.e., $\cdim_L P\leq 1$. Since $M$ is a representation of $P$ in the category of $\NP^r$-algebras,  it follows that it is a representation of $P$ in $\mathbf{Leib}$ as well, $P$ considered as the underlying Leibniz algebra. Now the result follows from long exact sequence $(A_2)$ in Theorem \ref{T:ex_seq}. Analogously for $P\in \NP^l$, where we apply $(A_1)$.

(ii) We apply again Proposition \ref{P:free_npr} and conclude that the underlying associative algebra of $P$ is free, from  which we have that $\cdim_H P\leq 1$. The result follows from the statement (i) of this Corollary, Corollary \ref{C:free_twozero} (ii) and the exact sequence $(G_2)$. Note that the statement (ii) for $n\geq 4$ follows from $(B_2)$ as well.
Analogously we obtain the equality $H^n(\cone(-\beta^{*})) = 0$, where we apply  exact sequence $(G_1)$ in Theorem \ref{T:ex_seq}.
\end{proof}

\begin{corollary} Let $P$ be an  $\AWB$. If\; $\cdim_H P\leq n$, $n\geq 1$, where $P$ is the underlying associative algebra, then
\[\cdim_{\AWB} P \leq n+1.\]
\end{corollary}
\begin{proof} Let $P$ be an  $\AWB^r$ or  an  $\AWB^l$. The results follow from  exact sequences $(C_{1,2})$ in Theorem \ref{T:ex_seq}.
Let $P$ be a left-right $\AWB$. Applying the result for $\AWB^r$  (or $\AWB^l$) for the underlying algebra $P$ as an $\AWB^r$ (resp.
as an $\AWB^l$), the result follows from exact sequences $(C, C')$
 in Theorem \ref{T:ex_seq}.
\end{proof}
\begin{corollary}
Let $P$ be a $\NP^{lr}$-algebra and $\cdim_H P\leq n$, $n\geq 2$. If $M$  is a representation of $P$, then we have:

\begin{itemize}
  \item[(i)] $H^{k+1}_{\NP^{lr}}(P,M) \approx  H^{k+1}\big(\cone(-\beta^*, -\beta'^*)\big),  \, H^{k+1}_{\NP^{l}}(P,M) \approx  H^{k+1}\big(\cone(-\beta^*)\big),$  $H^{k+1}_{\NP^{r}}(P,M) \approx  H^{k+1}\big(\cone(-\beta'^*)\big), \,  k > n$,
  where in the last two isomorphisms $P$ denotes the underlying $\NP^l$ and $\NP^r$-algebras of the given $\NP^{lr}$-algebra $P$, respectively;
  \item[(ii)] $H^{k+1}_{\NP^{lr}}(P,M) \approx   H^{k+1}_{\NP^{r}}(P,M) \approx   H^{k+1}_{\NP^{l}}(P,M) \, , k > n$, where $P$ in the last two right terms denotes the underlying $\NP^l$ and $\NP^r$-algebras of the given algebra $P$, respectively;
  \item[(iii)] $H^{k+1}_{\NP^{lr}}(P,M) \approx   H^{k+1}_L(P,M)  \, , k > n$, where on the right side $P$ denotes the underlying Leibniz algebra of the given algebra $P$.
\end{itemize}
\end{corollary}
\begin{proof}
(i) Follows from exact sequences $(B_1), (B_2)$ and $(B)$ in
Theorem \ref{T:ex_seq}. Analogously, for the proofs of (ii) and
(iii) we apply exact sequences $(D,D')$ and $(F)$, respectively.
Note that (iii) can be obtained as well by application of statement
(i) of this corollary and exact sequence $(H)$.
\end{proof}

The below stated corollaries are proved due to analogous arguments,
therefore the proofs are left to the reader.

\begin{corollary} Let $P$ be an  $\NP$-algebra and $M$ be a representation of $P$. If\; $\cdim_H P\leq n$  and $\cdim_L P\leq n$, $n\geq 2$, where $P$ denotes  the underlying associative and Leibniz algebras, respectively, then we have:
\begin{itemize}
\item[(i)] $\cdim_{\NP} P \leq n+1$;
\item[(ii)] $H^{k+1}\big(\cone(-\beta^*)\big)= H^{k+1}\big(\cone(-\beta'^*)\big)= H^{k+1}\big(\cone(-\beta \;^*, -\beta'^*)\big)=0 \; , \ k > n$.
\end{itemize}
\end{corollary}

\begin{corollary} Let $P$ be an $\NP$-algebra and $M$ be a representation of $P$. If\; $\cdim_L P\leq n$, where $P$ is the underlying Leibniz algebra, then we have:
\begin{itemize}

\item[(i)] $H^{k+1}_{\NP}(P,M) \approx  H^{k+1}_{\AWB}(P,M)  \; ,  \ k > n$;

\item[(ii)] $H^{k+1}\big(\cone(-\beta^*)\big) \ \approx \  H^{k+1}\big(\cone( -\beta'^*)\big) \ \approx H^{k+1}\big(\cone(-\beta^*, -\beta'^*)\big)\approx H^k_H(P,M^e), \;   \ k > n$.
\end{itemize}
\end{corollary}

\section*{Acknowledgments}
The authors are grateful  to referees for the helpful comments and suggestions.

The authors were supported by MICINN, grant MTM 2009-14464-C02 (Spain)
(European FEDER support included), and by Xunta de Galicia, grant Incite 09 207215PR. The second author is grateful to
Santiago de Compostela and Vigo Universities and to the Rustaveli
National Science Foundation  for financial support, grant GNSF/ST09
730 3-105.

\end{document}